\documentclass[12pt,a4paper]{amsart}

\usepackage[margin=3cm]{geometry}
\usepackage{t1enc}
\usepackage[utf8]{inputenc}
\usepackage{amsthm,amsmath,amssymb,amscd}
\usepackage{graphicx}
\usepackage{enumerate}
\usepackage{hyperref}
\usepackage{bm}
\usepackage{comment} 
\usepackage{amsfonts}
\usepackage{graphicx}
\usepackage{bm}
\usepackage{amsmath, amsthm, amssymb}
\usepackage{graphics}
\usepackage{hyperref}
 \usepackage{relsize}

\usepackage{bbm}

\theoremstyle{plain}

\makeatletter
\newcommand{\newreptheorem}[2]{\newtheorem*{rep@#1}{\rep@title}\newenvironment{rep#1}[1]{\def\rep@title{#2 \ref*{##1}}\begin{rep@#1}}{\end{rep@#1}}}
\makeatother

\newtheorem{theorem}{Theorem}[section]
\newtheorem*{theorem-non}{Theorem}
\newtheorem{lemma}[theorem]{Lemma}

\newreptheorem{theorem}{Theorem}
\newreptheorem{lemma}{Lemma}

\newtheorem{cor}[theorem]{Corollary}

\newtheorem{proposition}[theorem]{Proposition}

\theoremstyle{definition}

\newtheorem{definition}[theorem]{Definition}
\newtheorem{question}[theorem]{Question}
\newtheorem{conv}[theorem]{Convention}

\DeclareMathOperator{\mult}{mult}

\begin{document}
\title{Atoms of the matching measure}

\author[F. Bencs]{Ferenc Bencs}

\address{Central European University, Budapest\and \newline\indent Alfr\'ed R\'enyi Institute of Mathematics, Budapest}
\email{ferenc.bencs@gmail.com}

\author[A. M\'esz\'aros]{Andr\'as M\'esz\'aros}

\address{Central European University, Budapest\and\newline\indent Alfr\'ed R\'enyi Institute of Mathematics, Budapest}

\email{Meszaros\_Andras@phd.ceu.edu}



\maketitle

\begin{abstract}
We prove that the matching measure of an infinite vertex-transitive connected graph has no atoms. Generalizing the results of Salez, we show that for an ergodic non-amenable unimodular random rooted graph with uniformly bounded degrees, the matching measure has only finitely many atoms.  Ku and Chen proved the analogue of the Gallai-Edmonds structure theorem for non-zero roots of the matching polynomial for finite graphs. We extend their results for infinite graphs. We also show that the corresponding Gallai-Edmonds decomposition is compatible with the zero temperature monomer-dimer model.       
\end{abstract}




\section{Introduction}

First, we define the \emph{matching measure} of a \emph{rooted graph.} Fix a finite degree bound~$D$. Throughout the paper, we only consider graphs where the maximum degree is at most $D$. A rooted graph $(G,o)$ is a pair of a  connected (possibly infinite) graph $G$ and a distinguished vertex $o$ of $G$ called the root. Let $\mathcal{P}(o)$ be the set of finite paths in $G$ which start at $o$. The \emph{path tree} $T(G,o)$ of $G$ relative to $o$ has $\mathcal{P}(o)$ as its vertex set, and two
paths are adjacent if one is a maximal proper subpath of the other. For simplicity of notation, we  also use  $o$ to denote the path consisting of the single vertex $o$. 
The adjacency operator $A$ of $T(G,o)$ is a bounded self-adjoint operator on the Hilbert-space $\ell^2(\mathcal{P}(o))$ with norm at most~$D$.\footnote{In fact, it has norm at most $2\sqrt{D-1}$, but that will not be important for us.} The matching measure $\nu_{G,o}$ of the rooted graph $(G,o)$ is defined as the \emph{spectral measure} of $(T(G,o),o)$, that is, the unique probability measure on $[-D,D]$ with the property that for all $n\ge 1$, we have
\[\langle A^n\chi_o,\chi_o \rangle=\int_{-D}^D x^n d\nu_{G,o}(x).\]
Here $\chi_o$ is the characteristic vector of the one vertex path $o$.

So the matching measure is closely related to the spectral measure, that has been extensively studied for vertex transitive graphs. It was known that the spectral measure of a vertex transitive infinite graph can contain atoms. Then answering a long-standing open question, {\.Z}uk and Grigorchuk \cite{grigorchuk2001lamplighter} gave an example of an infinite connected vertex-transitive graph such that its spectral measure is purely atomic.  In contrast, we show that the matching measure of an infinite connected vertex-transitive graph has no atoms.

\begin{theorem}\label{thmtransitive}
Let $G$ be an infinite connected vertex-transitive graph, let $o$ be any vertex of it. Then $\nu_{G,o}$ has no atoms, that is,
\[\nu_{G,o}(\{\theta\})=0\]
for any $\theta\in \mathbb{R}$.
\end{theorem} 

The matching measure was introduced by Ab\'ert, Csikv\'ari, Frenkel and Kun~\cite{acsfk} as a tool to locally understand the matching polynomial of a finite graph. Given a finite graph with any root, its matching measure is supported on the roots of the \emph{matching polynomial}. Taking expectation over a uniform random root, we get the uniform probability measure on the roots of the matching polynomial.

Marcus,  Spielman and  Srivastava \cite{marcus2013interlacing} used the matching polynomial to construct bipartite Ramanujan graphs of all degrees.

\emph{Unimodular random rooted graphs} are natural generalizations of vertex-transitive graphs. We recall their  definition in  Section~\ref{unimod}. Note that if a random rooted graph is the local weak limit of finite graphs, then it is unimodular.

A graph $G$ is called \emph{amenable}, if for all $\varepsilon>0$, there is a finite subset $S$ of the vertices such that $|\partial S|<\varepsilon|S|$. 
Here $\partial S$ is the outer vertex boundary of~$S$, that is, $\partial S$ is the set of vertices which are not in $S$, but have a neighbor in~$S$.  

\begin{theorem}\label{thm25}
Let $(G,o)$ be an ergodic non-amenable random rooted graph with maximum degree at most $D$. Then $\mathbb{E}\nu_{G,o}$ has only finitely many atoms. \end{theorem}

We obtain Theorem \ref{thm25} from the following technical result. 

Fix a $\theta\in [-D,D]$. Given a graph $G$, and a vertex $u$ of $G$, we say that $u$ is \emph{$\theta$-essential} in $G$, if $\nu_{G,u}(\{\theta\})>0$. As we keep $\theta$ fixed, we will often simply use the term "essential" in place of the term "$\theta$-essential".

\begin{theorem}\label{thm2}
Let $(G,o)$ be a unimodular random rooted graph. Let $\mathfrak{S}$ be the set of essential vertices of $G$. 
For $o\in \mathfrak{S}$, let $\mathcal{C}_o$ be the connected component of $o$ in  the induced subgraph $G[\mathfrak{S}]$. Then $\mathcal{C}_o$ is finite with probability $1$, and
\[\mathbb{E}\nu_{G,o}(\{\theta\})\le \mathbb{E}\mathbbm{1}(o\in \mathfrak{S}) |\mathcal{C}_o|^{-1}-\mathbb{P}(o\in \partial \mathfrak{S}).\] 
Moreover, for $\theta=0$, we have an equality in the line above. 
\end{theorem}



The \emph{spectral measure}  of a rooted graph is a similar notion to the matching measure, that was under more intense research in the past decades. It is defined the same way as the matching measure, but we use the adjacency operator of $G$, instead of the adjacency operator of the path tree. We immediately see that the two definitions coincide when $G$ is a tree. Therefore,  results on the spectral measure of trees can also be viewed as results on the matching measure. For unimodular trees, the atom at~$0$ was analyzed in depth by  Bordenave,  Lelarge and  Salez~\cite{bls1}. Later they extended these results for the matching measure of general graphs~\cite{bls2}, although they did not use the matching measure terminology. Results on atoms other than $0$ were obtained by  Bordenave,  Sen and Vir\'ag~\cite{bsv}. Bordenave~\cite{bordenave} showed the existence of a continuous part of the spectral measure in certain cases. On trees, the behaviour of the spectral measure around $0$ was investigated by Coste and Salez~\cite{coste}.  

One of the main motivations of our work was the paper of Salez~\cite{salez}. He proved Theorem~\ref{thm25} and Theorem~\ref{thm2} in the special case of a unimodular tree. Actually, in  that more special setting, he proved a slightly stronger result, as he showed that the inequality in Theorem~\ref{thm2} can be replaced by an equality. Note that even for a unimodular random rooted graph $(G,o)$, it is generally not true that $(T(G,o),o)$ is unimodular. As an example, consider a vertex-transitive unimodular graph which is not a tree.  Thus, our theorems can not be deduced from the results of Salez in such a trivial way.  Most of the results above  rely on the tool of the \emph{Stieltjes transform}, and this will also be one of our main tools.


As mentioned before, the matching measure was introduced by Ab\'ert, Csikv\'ari, Frenkel and Kun~\cite{acsfk}. Their aim was to understand the asymptotic behaviour of the roots and coefficients of the matching polynomial in locally convergent sequences of finite graphs. See also \cite{acsh}. 
The \emph{matching polynomial} of a finite graph was defined by Heilmann and Lieb~\cite{heilmann} as follows. Given a finite graph $G$, we define its matching polynomial as
\[\mu(G,z)=\sum_{k\ge 0} (-1)^k p(k,G) z^{n-2k},\]
where $n$ is the number of vertices of $G$, and $p(k,G)$ is the number of matchings in $G$ with exactly $k$ edges. It was proved by Heilmann and Lieb~\cite{heilmann} that this polynomial has only real roots. Moreover, if all the degrees are at most $D$, then all the roots are contained in $[-2\sqrt{D-1},2\sqrt{D-1}]$.




Let $\nu_G$ be the uniform probability measure on the roots of the matching polynomial $\mu(G,z)$, that is,
\[\nu_G=\frac{1}{n}\sum_{i=1}^n \delta_{\lambda_i},\]
where $\lambda_1,\lambda_2,\dots,\lambda_n$ are the roots of $\mu(G,z)$ with multiplicities, and $\delta_{\lambda_i}$ is the Dirac-measure on $\lambda_i$. The measure $\nu_G$ can be disintegrated as
\[\nu_G=\mathbb{E} \nu_{G,o},\]
where $o$ is a uniform random vertex of $G$. 

For finite graphs, several matching related graph parameters can be recovered from the matching measure. For example, $\nu_G(\{0\})$ is equal to the proportion of vertices that are left uncovered by a maximum size matching of $G$. Furthermore, if $\gamma_G$ is the size of a uniform random matching, then
\[\frac{\mathbb{E}\gamma_G}{n}=\frac{1}{2}\int \frac{x^2}{1+x^2} d\nu_G(x)  =\frac{1}{2} \mathbb{E}\int\frac{x^2}{1+x^2} d\nu_{G,o}(x), \]
where the expectation is over a uniform random vertex $o$. Also, if $\mathbb{M}(G)$ is the set  of matchings in $G$, then   
\[\frac{\log |\mathbb{M}(G)|}{n}=\frac{1}{2}\int \log(1+x^2) d\nu_G(x)  =\frac{1}{2}\mathbb{E}\int\log(1+x^2) d\nu_{G,o}(x). \]

These formulas already suggest a way to extend these graph parameters from finite graphs to general unimodular random rooted graphs. These extensions are indeed meaningful, as they will be continuous with respect to the local weak convergence. This can be seen by using the fact that if a sequence of finite  graphs $(G_n)$ locally converges to a random rooted graph $(G,o)$, then the measures $\nu_{G_n}$ converge weakly to $\mathbb{E}\nu_{G,o}$. We will not give more details, the interested reader should consult the papers~\cite{acsfk,acsh}.    

Using the matching measure, Csikv\'ari \cite{csikvari2016matchings} proved that if we restrict our attention to vertex-transitive bipartite graphs, then the number of perfect matchings is also well-behaved with respect to local weak convergence.  Another successful application of these notions is the proof of the Friedland's Lower Matching Conjecture by Csikv\'ari \cite{csikvari2017lower}  
that provides a lower bound on the number of matchings of a given size in finite $d$-regular bipartite graphs. 

The other main motivation for our results is the paper of Ku and  Chen~\cite{kucheng}. Building on the work of Godsil~\cite{godsil}, they proved the analogue of the \emph{Gallai-Edmonds structure theorem} for nonzero roots of the matching polynomial for finite graphs. 

First, let us recall a few classical combinatorial theorems from the matching theory of finite graphs.   Classically, a vertex $u$ of a  finite graph $G$ is called \emph{essential}, if there is a maximum size matching in $G$ which leaves $u$ uncovered. Note that this definition coincides with our definition of a $0$-essential vertex, because $\nu_{G,u}(\{0\})$ is the probability that a uniform random maximum size matching leaves $u$ uncovered. 

A graph is called \emph{factor-critical} if it has only essential vertices.

\begin{lemma}[Gallai~\cite{gallai}]
Let $G$ be a finite, connected factor-critical graph. Then each maximum size matching leaves exactly one vertex uncovered. 
\end{lemma}

Note that the conclusion of the lemma above can be rephrased as \[\sum_{u\in V(G)}\nu_{G,u}(\{0\})=1.\]

\begin{samepage}
\begin{theorem}[Gallai-Edmonds structure theorem~\cite{gallai, edmonds}\footnote{See also \cite[Theorem~3.2.1]{lovasz2009matching} for a formulation much closer to ours. }]
Given a finite graph~$G$, let $D$ be the set of essential vertices in $G$. Let $A=\partial D$, and let $C=V(G)-D-A$.
Then
\begin{enumerate}[(a)]
\item All the components of the induced subgraph $G[D]$ are factor-critical.
\item Any vertex of the induced subgraph $G[C]$ is non-essential (in the graph $G[C]$).
\item Let $X$ be a subset of $A$, then there are at least $|X|+1$ connected components in $G[D]$ which are connected to a vertex in $X$ in the graph $G$. 
\item Moreover, each maximum size matching $M$ must satisfy the following properties:
\begin{itemize}
    \item The vertices in $A\cup C$ are all covered by $M$.  
    \item Every vertex in $A$ is matched with a vertex in $D$.
    \item Every connected component of $G[D]$ contains at most one vertex not covered by $M$.
    \item Every connected component of $G[D]$ contains at most one vertex which is matched with a vertex in $A$.
    \item The number of vertices that are uncovered by the matching $M$ is equal to the number of connected components in $G[D]$ minus $|A|$. 
\end{itemize}
\end{enumerate}
\end{theorem}
\end{samepage}
Now we state the generalizations of these theorems by Ku and Chen~\cite{kucheng}. As an obvious generalization of the factor-critical graphs above, we call a graph \emph{$\theta$-critical} if all the vertices are $\theta$-essential. 

\begin{lemma}[The analogue of Gallai's Lemma by Ku and Chen~\cite{kucheng}] Let $G$ be a finite, connected  $\theta$-critical graph. Then the root $\theta$ has multiplicity $1$ in the matching polynomial of $G$. 
\end{lemma}

This lemma has the following corollary.
\begin{cor}[Ku, Chen~\cite{kucheng}]
If $G$ is a finite, connected vertex-transitive  graph, then the matching polynomial of $G$ has only simple roots. 
\end{cor}

The statement above can be viewed as the finite counterpart of our Theorem~\ref{thmtransitive}. Note that the special case of Theorem~\ref{thmtransitive} when $G$ can be approximated locally by finite transitive graphs was proved by Ab\'ert, Csikv\'ari and  Hubai~\cite{acsh}.

\begin{theorem}[The analogue of the Gallai-Edmonds structure theorem by Ku and Chen~\cite{kucheng}]
Let $G$ be a finite graph. Let $D$ be the set of $\theta$-essential vertices in $G$. Let $A=\partial D$ and $C=V(G)-D-A$. Then 
\begin{enumerate}[(a)]
\item All the components of the induced subgraph $G[D]$ are $\theta$-critical.
\item Any vertex of the induced subgraph $G[C]$ is not $\theta$-essential (in the graph $G[C]$).
\item The multiplicity of the root $\theta$ in the matching polynomial of $G$ is given by the number of connected components in $G[D]$ minus $|A|$.
\end{enumerate}
\end{theorem}

We generalize these results for infinite graphs. Recall that we always assume that the maximum degree of our graphs are finite. 

\begin{lemma}[The analogue of Gallai's lemma]\label{egyszerugyok}
Let $G$ be a connected (possibly infinite) $\theta$-critical graph. Then
\[\sum_{u\in V(G)} \nu_{G,u}(\{\theta\})=1.\]
\end{lemma}
In the unimodular case, we have an even stronger result.
\begin{theorem}\label{lemmafinite}
Let $(G,o)$ be a unimodular random rooted graph. If $G$ is $\theta$-critical with probability $1$, then $G$ is finite with probability $1$.
\end{theorem}
We have the following analogue of the Gallai-Edmonds structure theorem.

\begin{theorem}\label{thm:GE_theta}
Given a finite graph $G$, let $D$ be the set of $\theta$-essential vertices in $G$. Let $A=\partial D$, and let $C=V(G)-D-A$.
Then
\begin{enumerate}[(a)]
\item All the components of the induced subgraph $G[D]$ are $\theta$-critical.
\item Let $X$ be a subset of $A$, then there are at least $|X|+1$ connected components in $G[D]$ which are connected to a vertex in $X$ in the graph $G$. 
\end{enumerate}
\end{theorem}
Note that in the unimodular case, in addition to the facts above, Theorem~\ref{thm2} also provides an upper bound on the atom $\mathbb{E}\nu_{G,o}(\{\theta\})$, and it also gives the finiteness of the connected components of $G[D]$. 

For $\theta=0$, we can prove even more. Given a graph $G$, we can define certain random matchings  of $G$ which are  called Boltzmann random matchings at temperature zero.  For a finite graph, a Boltzmann random matching at temperature zero  is simply a uniform random maximum size matching. For infinite graphs, the definition is more involved. See Section~\ref{sec:monomerdimer} for details. 

\begin{theorem}[Gallai-Edmonds structure theorem for the monomer-dimer model]\label{thm:GE_0}
Let $G$ be a possibly infinite graph, let $D$ be the set of $0$-essential vertices in $G$. Let $A=\partial D$, and let $C=V(G)-D-A$.
Then
\begin{enumerate}[(a)]
\item\label{GallaiEdmondsa} All the components of the induced subgraph $G[D]$ are 0-critical.
\item\label{GallaiEdmondsb} Let $X$ be a subset of $A$, then there are at least $|X|+1$ connected components in $G[D]$ which are connected to a vertex in $X$ in the graph $G$. 
\item\label{GallaiEdmondsc} Let $\mathcal{M}$ be  a Boltzmann random matching at temperature zero. Then $\mathcal{M}$ has the following properties with probability $1$: 
\begin{itemize}
    \item The vertices in $A\cup C$ are all covered by $\mathcal{M}$.  
    \item Every vertex in $A$ is matched with a vertex in $D$.
    \item Every connected component of $G[D]$ contains at most one vertex not covered by $\mathcal{M}$.
    \item Every connected component of $G[D]$ contains at most one vertex which is matched with a vertex in $A$.
\end{itemize}
\item\label{GallaiEdmondsd} If $(G,o)$ is a unimodular random rooted graph, then
\[\mathbb{E}\nu_{G,o}(\{0\})=\mathbb{E}\mathbbm{1}(o\in D) |\mathcal{C}_o|^{-1}-\mathbb{P}(o\in A).\]
Here, for  $o\in D$, $\mathcal{C}_o$ is the connected component of $o$ in the graph $G[D]$.
\end{enumerate}
\end{theorem}

Combining part \eqref{GallaiEdmondsc} of Theorem \ref{thm:GE_0} with Theorem \ref{thmtransitive}, we obtain another proof of the following theorem of Cs\'oka and Lippner \cite{csoka2016invariant}.
\begin{cor}
Every infinite vertex-transitive connected bounded degree graph has a perfect matching.
\end{cor}

Note that for an infinite  graph $G$ and vertex $u$, it is not true that $u$ is $0$-essential in our spectral sense if and only if there is a maximum matching that does not cover~$u$. In fact, it is not even clear what we mean by maximum matching. Nevertheless, one could make a precise definition of a maximum matching in an infinite graph, and then obtain an alternative definition of essential vertices. However, this definition does not coincide with ours, and it gives a different theory of the Gallai-Edmonds decomposition. See the work of Bry and  Las Vergnas~\cite{bry}.  

This paper is organized as follows. In Section~\ref{sec:pre}, we define the Stieltjes-transform of the matching measure, the monomer dimer model and unimodular random rooted graphs. Following the terminology of Godsil~\cite{godsil}, we also define three types of vertices: essential, neutral and positive. In Section~\ref{sec:delver}, we investigate the effect of deleting different types of vertices. In Section~\ref{sec:GGE}, relying on these results, we  prove our general version of the Gallai-lemma and the Gallai-Edmonds structure theorem, then we obtain Theorem~\ref{thmtransitive} as an easy corollary. In Section~\ref{sec::monomerdimer}, we prove Theorem~\ref{thm:GE_0}.   In Section~\ref{sec:unimodsec}, we prove Theorem~\ref{thm25} and Theorem~\ref{thm2}. 
Several open questions and additional results are listed in Section~\ref{sec:open}. 

\textbf{Acknowledgements.} The authors are grateful to Mikl\'os Ab\'ert and P\'eter \break Csikv\'ari for their comments on the manuscript. 
F. B.  has received funding from the European Research Council under the European Union’s Seventh Framework Programme (FP7/2007-2013) / ERC grant agreement n$^\circ$ 617747; The research was partially supported by the MTA R\'enyi Institute Lend\"ulet Limits of Structures Research Group. A. M. was partially
supported by the ERC Consolidator Grant 648017.

\section{Preliminaries}\label{sec:pre}

\subsection{Unimodular random rooted graphs}\label{unimod}

Strictly speaking, we need to define a measurable structure on the space of (isomorphism classes of) rooted graphs to be able to speak about random rooted graphs. However, we omit these rather technical details, and we refer the reader to the paper of Aldous and Lyons~\cite{aldous2007processes} instead. 

A \emph{bi-rooted graph} is a triple $(G,x,y)$, where $G$ is a connected graph,  $x$ and $y$ are two vertices of $G$. The space of  (isomorphism classes of) bi-rooted graphs can be endowed with a measurable structure. We again omit the details. 

A random rooted graph $(G,o)$ is called \emph{unimodular} if it satisfies the so-called \break \emph{Mass-Transport Principle}, that is, for any non-negative measurable function $f$ defined on the space of bi-rooted graphs, we have
\[\mathbb{E}\sum_{v\in V(G)} f(G,o,v)=\mathbb{E} \sum_{v\in V(G)} f(G,v,o).\]

A measurable function $h$ defined on the space of rooted graphs is called \emph{invariant} if $h(G,o)$ only depends on $G$, but not on the chosen root $o$. A unimodular random rooted graph $(G,o)$ is called \emph{ergodic} if for all invariant measurable functions $h$, we have that $h(G,o)$ is constant almost surely. 

Sometimes, we will need a bit more general version of the notion of unimodularity. Let $\Xi$ be a complete separable metric space called the mark space. A  rooted decorated graph is a triple $(G,m,o)$, where $(G,o)$ is a rooted graph and $m$ is a map from $V(G)$ to $\Xi$. We define bi-rooted decorated graphs in an analogous way.  A random  rooted decorated graph $(G,m,o)$ is called \emph{unimodular} if for any non-negative measurable function $f$ defined on the space of bi-rooted decorated graphs, we have
\[\mathbb{E}\sum_{v\in V(G)} f(G,m,o,v)=\mathbb{E} \sum_{v\in V(G)} f(G,m,v,o).\]
Again, we omitted the details of measurability.

For example, this general definition allows us to speak about the unimodularity of a random tuple  $(G,N,\ell,o)$, where $(G,o)$ is a random rooted graph, $N$ is subset of $V(G)$ and $\ell:V(G)\to [0,1]$ is a labeling of the vertices. Indeed, if we let $\Xi=\{0,1\}\times [0,1]$, then the pair $(N,\ell)$ can be encoded as a map $m:V(G)\to \Xi$, where $m(v)=(\mathbbm{1}(v\in N),\ell(v))$.

The next two lemmas  are typical applications of the Mass-Transport Principle. These statements are well-known, but we give the proofs for the reader's convenience. 

\begin{lemma}\label{asfinite}
Let $(G,N,o)$ be a unimodular random triple, where $N$ is a non-empty finite subset of $V(G)$. Then $G$ is finite with probability $1$.  
\end{lemma}
\begin{proof}
For the sake of contradiction, assume that $G$ is infinite with positive probability. We can choose a $0<k<\infty$ such that \[\mathbb{P}(|V(G)|=\infty\text{ and }|N|=k)>0.\] Let us define
\[f(G,N,x,y)=\mathbbm{1}(|V(G)|=\infty, \quad |N|=k\text{ and }y\in N).\]
Then \[\mathbb{E}\sum_{v\in V(G)}f(G,N,o,v)=k\cdot \mathbb{P}(|V(G)|=\infty\text{ and }|N|=k),\] which is a positive finite number.  On the other hand, \[\mathbb{E}\sum_{v\in V(G)}f(G,N,v,o)=\mathbb{P}(|V(G)|=\infty, \quad |N|=k\text{ and }o\in N)\cdot \infty,\] which is either $0$ or infinite. This gives us a contradiction.
\end{proof}
\begin{lemma}\label{perc}
Let $(G,N,o)$ be a unimodular random triple, where $N$ is a subset of~$V(G)$. Assume that $\mathbb{P}(o\in N)>0$. For $x\in N$, let $N_x$ be the set of vertices of the connected component of $x$ in the induced subgraph $G[N]$. Let $(G',N',o')$ have the same law as $(G,N,o)$ conditioned on $o\in N$. Then the random rooted graph $(G'[N_{o'}'],o')$ is unimodular.
\end{lemma}
\begin{proof}
Let $f'$ be any non-negative measurable function on the space of bi-rooted graphs. We need to prove that
\[\mathbb{E}\sum_{v'\in N_{o'}'}f'(G'[N_{o'}'],o',v')=\mathbb{E}\sum_{v'\in N_{o'}'}f'(G'[N_{o'}'],v',o').\]
Let use define
\[f(G,N,x,y)=\mathbbm{1}(x,y\in N\text{ and }N_x=N_y) f'(G[N_x],x,y).\]
Note that
\[\mathbb{E}\sum_{v\in V(G)} f(G,N,o,v)=\mathbb{P}(o\in N)\cdot  \mathbb{E}\sum_{v'\in N_{o'}'}f'(G'[N_{o'}'],o',v')\]
and 
\[\mathbb{E}\sum_{v\in V(G)} f(G,N,v,o)=\mathbb{P}(o\in N)\cdot  \mathbb{E}\sum_{v'\in N_{o'}'}f'(G'[N_{o'}'],v',o').\]
Therefore, from the Mass-Transport Principle
\[\mathbb{P}(o\in N)\cdot\mathbb{E}\sum_{v'\in N_{o'}'}f'(G'[N_{o'}'],o',v')=\mathbb{P}(o\in N)\cdot\mathbb{E}\sum_{v'\in N_{o'}'}f'(G'[N_{o'}'],v',o').\]
Since $\mathbb{P}(o\in N)>0$, the statement follows. 
\end{proof}

\subsection{Spectral definitions}

Let $G$ be a (possibly infinite) connected graph with uniform degree bound $D$. Given a vertex $u$ of $G$, let $\mathcal{P}(u)=\mathcal{P}_G(u)$ be the set of finite paths in $G$ which start at $u$. The path tree $T(G,u)$ of G relative to $u$ has $\mathcal{P}(u)$ as its vertex set, and two
paths are adjacent if one is a maximal proper subpath of the other. For simplicity of notation, we will also use  $u$ to denote the path consisting of the single vertex $u$. Note that each path in
$\mathcal{P}(u)$ determines a path starting with $u$ in $T(G,u)$  with the same length. The adjacency operator $A$ of $T(G,u)$ is a bounded self-adjoint operator on the Hilbert-space $\ell^2(\mathcal{P}(u))$ with norm at most $D$. For a path $P\in\mathcal{P}(u)$, let $\chi_P\in \ell^2(\mathcal{P}(u))$ be its characteristic vector. Let $\pi$ be the projection valued measure corresponding to the operator $A$. Let $H\subset \mathbb{C}$ be the open upper half-plane. For a path $P\in \mathcal{P}(u)$ and $z\in H$, we define
\[s_{G,P}(z)=\langle (zI-A)^{-1}\chi_u,\chi_P\rangle.\] 

We define the signed measure $\nu_{G,P}$ on $[-D,D]\subset  \mathbb{R}$, by setting \break $\nu_{G,P}(E)=\langle \pi(E)\chi_u,\chi_P\rangle$ for all measurable subset $E$ of $[-D,D]$. Note that $\nu_{G,P}$ has total variation at most~$1$. Moreover,  $\nu_{G,u}$ is a probability measure. From the Spectral-theorem, we have
\[s_{G,P}(z)=\int_{-D}^D (z-x)^{-1} d\nu_{G,P}(x),\]
that is, $s_{G,P}(z)$ is the  \emph{Stieltjes-transform} of $\nu_{G,P}$.

Later, in Lemma~\ref{reverse}, we will prove that $s_{G,P}(z)$ is the same for both orientation of the path $P$. However, it is not at all clear at this point.

We recall the second resolvent identity.
\begin{proposition}
Let $A$ and $B$ be two bounded self-adjoint operators on the same Hilbert-space. For any $z\in H$, we have
\begin{align*}
(zI-A)^{-1}-(zI-B)^{-1}&=(zI-A)^{-1}(A-B)(zI-B)^{-1}\\
&= (zI-B)^{-1}(A-B)(zI-A)^{-1}.\qedhere
\end{align*}
\end{proposition} 

\begin{lemma}
Let $P=(p_0,p_1,\dots,p_k)$ be a path with at least one edge. Let $P'=(p_1,p_2,\dots, p_k)$. Then
\[s_{G,P}(z)=s_{G,p_0}(z)s_{G-p_0,P'}(z).\] 
\end{lemma}

\begin{proof}
Let $\mathcal{P}_1\subset  \mathcal{P}(p_0)$ be the set of paths such that their first edge is $(p_0,p_1)$. Let $T_1$ be the subtree of $T(G,p_0)$ induced by the set of vertices $\mathcal{P}_1$. Let $A$ be the adjacency operator of $T(G,p_0)$.  Let $B$ be the adjacency operator of $T_1$. It still can be considered as an operator on $\ell^2(\mathcal{P}(p_0))$. Note that $T_1$ is isomorphic to $T(G-p_0,p_1)$. Thus, if $A_1$ is the adjacency operator of $T(G-p_0,p_1)$, then 
\[s_{G-p_0,P'}(z)=\langle (zI-A_1)^{-1} \chi_{p_1},\chi_{P'}\rangle=\langle (zI-B)^{-1} \chi_{(p_0,p_1)},\chi_{P}\rangle.\]
From the second resolvent identity, we have
\begin{align*}
\langle (zI-A)^{-1} \chi_{p_0},\chi_{P}\rangle-&\langle (zI-B)^{-1} \chi_{p_0},\chi_{P}\rangle\\
&=\langle (zI-B)^{-1}(A-B)(zI-A)^{-1} \chi_{p_0},\chi_{P}\rangle.
\end{align*}
Observe that the left hand side is equal to 
\begin{align*}
\langle (zI-A)^{-1} \chi_{p_0},\chi_{P}\rangle-&\langle (zI-B)^{-1} \chi_{p_0},\chi_{P}\rangle \\
&=\langle (zI-A)^{-1} \chi_{p_0},\chi_{P}\rangle-\langle z^{-1} \chi_{p_0},\chi_{P}\rangle\\
&=s_{G,P}(z).
\end{align*}
Therefore,
\begin{align*}
    s_{G,P}(z)&=\langle (zI-B)^{-1}(A-B)(zI-A)^{-1} \chi_{p_0},\chi_{P}\rangle\\
    &=\langle (A-B)(zI-A)^{-1} \chi_{p_0},(\overline{z}I-B)^{-1}\chi_{P}\rangle.
\end{align*}
Note that $(\overline{z}I-B)^{-1}\chi_{P}$ is supported on $\mathcal{P}_1$. Moreover, $(A-B)(zI-A)^{-1} \chi_{p_0}$ supported on $(\mathcal{P}(p_0)\backslash \mathcal{P}_1)\cup \{(p_0,p_1)\}$. Since $(p_0,p_1)$ is the only common element of these supports, we have
\begin{align*}
\langle (A-B)(zI-&A)^{-1} \chi_{p_0},(\overline{z}I-B)^{-1}\chi_{P}\rangle\\
&=\langle (A-B)(zI-A)^{-1} \chi_{p_0},\chi_{(p_0,p_1)}\rangle\overline{\langle (\overline{z}I-B)^{-1}\chi_{P},\chi_{(p_0,p_1)}\rangle}\\
&=\langle (zI-A)^{-1} \chi_{p_0},(A-B)\chi_{(p_0,p_1)}\rangle\langle(zI-B)^{-1} \chi_{(p_0,p_1)},\chi_{P}\rangle\\
&=\langle (zI-A)^{-1} \chi_{p_0},\chi_{p_0}\rangle\langle(zI-A_1)^{-1} \chi_{p_1},\chi_{P'}\rangle\\
&=s_{G,p_0}(z)s_{G-p_0,P'}.
\end{align*}
\end{proof}

Iterating the previous lemma, we get the following lemma.

\begin{lemma}\label{kifejt}
Let $P=(p_0,p_1,\dots,p_k)$ be a path.  Then
\[\pushQED{\qed} s_{G,P}(z)=\prod_{i=0}^k s_{G-\{p_0,\dots,p_{i-1}\},p_i}(z).\qedhere\popQED\] 
\end{lemma}

\begin{lemma}\label{limit_it}
Let $\theta$ be a real number. For any path $P$, we have
\[\nu_{G,P}(\{\theta\})=\lim_{t\to 0} it \cdot s_{G,P}(\theta+it).\]
In particular, for any vertex $u$, we have
\[\nu_{G,u}(\{\theta\})=\lim_{t\to 0} it \cdot s_{G,u}(\theta+it).\]

\end{lemma} 
\begin{proof}
Recall that
\[it\cdot s_{G,P}(\theta+it)=\int_{-D}^D \frac{it}{\theta-x+it}d\nu_{G,P}(x).\]

Note that 
\[\lim_{t\to 0}\frac{it}{\theta-x+it} =\begin{cases}
1&\text{for }x=\theta,\\
0&\text{for }x\neq\theta.
\end{cases}
\]

Since $\left|\frac{it}{\theta-x+it}\right|\le 1$ and $\nu_{G,P}$ has finite total variation, we can use the dominated convergence theorem to obtain that
\begin{align*}
\lim_{t\to 0} it\cdot s_{G,P}(\theta+it)&=\lim_{t\to 0}\int_{-D}^D \frac{it}{\theta-x+it} d\nu_{G,P}(x)\\&=\int_{-D}^D \lim_{t\to 0} \frac{it}{\theta-x+it} d\nu_{G,P}(x)\\&=\nu_{G,P}(\theta).    \end{align*}
\end{proof}

\subsection{Finite graphs and the matching polynomial}

Assume that $G$ is a finite graph. Let $p(k,G)$ be the number of matchings in $G$ with exactly $k$ edges. Let $n$ be the number of vertices of $G$. The matching polynomial of $G$ is defined as
\[\mu(G,z)=\sum_{k\ge 0} (-1)^k p(k,G) z^{n-2k}.\]
\begin{lemma}\label{sGminthanyados}
For any finite graph $G$, a vertex $u$ of $G$ and $z\in H$, we have
\[s_{G,u}(z)=\frac{\mu(G-u,z)}{\mu(G,z)}.\]
\end{lemma}
\begin{proof}
Let $T=T(G,u)$ be the path tree of $G$ relative to $u$. As before, let $A$ be the adjacency matrix of $T$. Recall that $\mu(T,z)=\det(zI-A)$ since $T$ is a tree, see for example \cite{godsil}. Let $A_0$ be the matrix obtained from $A$ by removing the row and column corresponding to $u$. From Cramer's rule, we have
\[s_{G,u}(z)=\langle (zI-A)^{-1} \chi_u,\chi_u\rangle =\frac{\det (zI-A_0)}{\det (zI-A)}=\frac{\mu(T-u,z)}{\mu(T,z)}.\]
From~\cite[Corollary 2.3]{godsil}, we have that
\[\frac{\mu(T-u,z)}{\mu(T,z)}=\frac{\mu(G-u,z)}{\mu(G,z)}.\]
\end{proof}

Even more generally, we have:
\begin{lemma}\label{pathremove}
For any finite graph $G$, a path $P$ of $G$, and $z\in H$, we have
\[s_{G,P}(z)=\frac{\mu(G-P,z)}{\mu(G,z)}.\]
\end{lemma}
\begin{proof}
Let $P=(p_0,p_1,\dots,p_k)$. From Lemma~\ref{kifejt} and the previous lemma, we have
\begin{multline*}s_{G,P}(z)=\prod_{i=0}^k s_{G-\{p_0,\dots,p_{i-1}\},p_i}(z)=\prod_{i=0}^k \frac{\mu(G-\{p_0,\dots,p_i\},z)}{\mu(G-\{p_0,\dots,p_{i-1}\},z)}=\frac{\mu(G-P,z)}{\mu(G,z)}.\end{multline*}
\end{proof}

For a finite graph $G$, let $\mult(\theta,G)$ be the multiplicity of the root $\theta$ in the matching polynomial $\mu(G,z)$. 
\begin{lemma}
For a finite graph $G$, we have
\[\sum_{u\in V(G)} \nu_{G,u}(\{\theta\})=\mult(\theta,G).\]
Consequently, we have
\[\nu_G=\mathbb{E}\nu_{G,o},\]
where the expectation is over a uniform random vertex $o$.
\end{lemma}
\begin{proof}
From~\cite[Theorem 2.1 (d)]{godsil}, we have
\[\sum_{u\in V(G)} \mu(G-u,z)=\mu'(G,z).\]
Therefore,
\begin{align*}
\sum_{u\in V(G)} \nu_{G,u}(\{\theta\})&=\lim_{t\to 0} it\sum_{u\in V(G)} s_{G,u}(\theta+it)\\&=\lim_{t\to 0} it\sum_{u\in V(G)} \frac{\mu(G-u,\theta+it)}{\mu(G,\theta+it)}\\&=\lim_{t\to 0} \frac{it\mu'(G,\theta+it)}{\mu(G,\theta+it)}\\&=\mult(\theta,G).
\end{align*}
\end{proof}

\subsection{Exhaustion by finite graphs}

Let $G$ be a countable connected graph with uniform degree bound $D$. Let $V_1\subseteq  V_2\subseteq  V_2\subseteq \dots$ be an infinite sequence of finite subsets of the vertex set $V(G)$ such that $\cup_{i=1}^\infty V_i=V(G)$. Let $G_i$ be the subgraph of $G$ induced by $V_i$. We call the sequence $(G_i)$ an \emph{exhaustion} of the graph $G$.

\begin{lemma}\label{slimit}
For any vertex $u\in V_1$, the measures $\nu_{G_i,u}$ weakly converge to $\nu_{G,u}$. 

Consequently, for any $z\in H$, we have
\[\lim_{i\to\infty} s_{G_i,u}(z)=s_{G,u}(z).\]
\end{lemma}  
\begin{proof}
Since the supports of the measure are contained in $[-D,D]$, it is enough to prove that for any $n\ge 0$, we have
\[\lim_{i\to\infty} \int_{-D}^D x^n d\nu_{G_i,u}(x)=\int_{-D}^D x^n d\nu_{G,u}(x).\]
Let $A_i$ be the adjacency operator of $T(G_i,u)$. Since $\int_{-D}^D x^n d\nu_{G_i,u}(x)=\langle A_i^n\chi_u,\chi_u\rangle$, we see that $\int_{-D}^D x^n d\nu_{G_i,u}(x)$ is the number of walks of length $n$  from $u$ to $u$ in the graph $T(G_i,u)$. If $i$ is large enough, then $V_i$ contains all the vertices that are at most distance $n$ from $u$ in the graph $G$. It is easy to see from the  walk counting  that in this case $\int_{-D}^D x^n d\nu_{G_i,u}(x)=\int_{-D}^D x^n d\nu_{G,u}(x)$, so the statement follows.
\end{proof}

\begin{lemma}\label{reverse}
Let $P=(p_0,p_1,\dots,p_k)$ be a path. Let $P'=(p_k,p_{k-1},\dots,p_0)$ be its reversed path. Then
\[s_{G,P}(z)=s_{G,P'}(z)\]
for any $z\in H$.

Moreover, \[\nu_{G,P}(\{\theta\})=\nu_{G,P'}(\{\theta\}).\] 

\end{lemma}
\begin{proof}
We start by the first statement. For a finite graph $G$, it is clear from Lemma~\ref{pathremove}. Assume that $G$ is infinite. Take an exhaustion of $G$, such that $P$ contained in $V_1$. The statement will follows from the finite case, if we prove that for all $z\in H$, we have
\[\lim_{i\to\infty} s_{G_i,P}(z)=s_{G,P}(z).\]
This is indeed true, as we see form Lemma~\ref{kifejt} and Lemma~\ref{slimit}.  

The second statement follows from the first one, and Lemma~\ref{limit_it}.
\end{proof}

\begin{lemma}\label{rekurz}
For any vertex $u$ and $z\in H$, we have
\[1=z\cdot s_{G,u}(z)-\sum_{v\sim u} s_{G,(u,v)}(z),\]
where the summation is over the neighbors $v$ of $u$.  
\end{lemma}
\begin{proof}
This is well-known for finite graphs, see for example \break\cite[Theorem~2.1~(c)]{godsil}. By exhaustion, it is also true for infinite graphs.
\end{proof}

Let $K=\{k_1,k_2,\dots,k_m\}$ be a finite set of vertices. Let us define
\[s_{G,K}(z)=\prod_{i=1}^m s_{G-\{k_1,\dots,k_{i-1}\},k_i}(z).\]
Note that we already defined $s_{G,K}$, where $K$ is a path, Lemma~\ref{kifejt} shows that this definition is consistent with our previous definition. The next lemma shows that this is a well-defined notion. 
\begin{lemma}\label{felcserel}
The value of $s_{G,K}(z)$ does not depend on the ordering of the elements of $K$.
\end{lemma}
\begin{proof}
It is enough to prove for finite graphs, because infinite graphs can be handled by exhaustion. For a finite graph similarly as in Lemma~\ref{pathremove}, we have
\[\prod_{i=1}^m s_{G-\{k_1,\dots,k_{i-1}\},k_i}(z)=\frac{\mu(G-K,z)}{\mu(G,z)},\]
which  clearly does not depend on the ordering of the elements of $K$. 
\end{proof}

\subsection{The monomer-dimer model}\label{sec:monomerdimer}

First, let $G$ be a finite graph. The set of matchings in $G$ is denoted by $\mathbb{M}(G)$. For any $t>0$, we consider a random matching $\mathcal{M}_G^t\in \mathbb{M}(G)$ such that for any $M\in \mathbb{M}(G)$, we have
\[\mathbb{P}(\mathcal{M}_G^t=M)=\frac{t^{|V(G)|-2|M|}}{P_G(t)},\]
where $P_G(t)=\sum_{M\in\mathbb{M}(G)} t^{|V(G)|-2|M|}=(-i)^n\mu(G,it)$. The random matching $\mathcal{M}_G^t$ is 
called a Boltzmann random matching at temperature $t$. 

Now, we extend these definitions for an infinite countable graph $G$ with maximum degree at most $D$. Let $V_1\subseteq  V_2\subseteq  V_2\subseteq \dots$ be an infinite sequence of finite subsets of the vertex set $V(G)$ such that $\cup_{i=1}^\infty V_i=V(G)$. Let $G_i$ be the subgraph of $G$ induced by $V_i$. We call the sequence $(G_i)$ an exhaustion of the graph $G$. Since $E(G_n)\subseteq  E(G)$, we can consider $\mathcal{M}_{G_n}^t$ as a random matching of $G$. 

\begin{lemma}[\cite{bls2}]\label{postemp}
Fix any $t>0$. The random matchings $\mathcal{M}_{G_n}^t$ converge in law to a random matching $\mathcal{M}_G^t$ as $n\to \infty$. The law of $\mathcal{M}_G^t$ does not depend on the chosen exhaustion. 

Moreover, for any finite subset  $X$ of $V(G)$, we have
\[\mathbb{P}(\mathcal{M}_G^t\text{ leaves uncovered all the vertices in }X)=(it)^{|X|}s_{G,X}(it),\]
and for any finite matching $M\in \mathbb{M}(G)$, we have
\[\mathbb{P}(M\subseteq  \mathcal{M}_G^t)=(-1)^{|M|} s_{G,V(M)}(it),\]
where $V(M)$ is the set of vertices covered by $M$.
\end{lemma}

A random matching $\mathcal{M}$ of $G$ is called a Boltzmann random matching at temperature zero, if there is a sequence $t_1,t_2,\dots$ of positive reals tending to zero such that the random matchings $\mathcal{M}_G^{t_n}$ converge in law to $\mathcal{M}$. In other words, for any finite matching $M\in \mathbb{M}(G)$, the limit  $\lim_{n\to\infty} s_{G,V(M)}(it_n)$ exists, and
\[\mathbb{P}(M\subseteq  \mathcal{M})=(-1)^{|M|} \lim_{n\to\infty} s_{G,V(M)}(it_n).\]

We do not know whether the random matchings $\mathcal{M}_G^t$  converge in law as $t$ tends to  zero. This would imply that the distribution of  a Boltzmann random matching at temperature zero is uniquely determined. See Question~\ref{qmonomerdimer} in Section~\ref{sec:open} for further discussion. Anyway,  by a standard  compactness argument, we see that there is always at least one Boltzmann random matching at temperature zero. For a finite graph $G$, the random matchings $\mathcal{M}_G^t$  converge in law to a uniform random maximum size matching as $t\to 0$.

Moreover, several observables have the same distribution for all the  Boltzmann random matchings at temperature zero. For example, the distribution of the vertices that are  covered by the matching is the same for every  Boltzmann random matching at temperature zero as the next lemma shows.

\begin{lemma}[\cite{bls2}]\label{welld}
Let $\mathcal{M}$ be any Boltzmann random matching at temperature zero. For  
any finite subset $X$ of $V(G)$, we have
\[\mathbb{P}(\mathcal{M}\text{ leaves uncovered all the vertices in }X)=\lim_{t\to 0} (it)^{|X|}s_{G,X}(it).\]
In particular, the limit above exists.
\end{lemma}

\subsection{Essential, neutral and positive vertices}

Fix a real $\theta$. A vertex $u$ in $G$ is called \emph{essential}, if $\nu_{G,u}(\{\theta\})>0$. Or more generally a path $P$ is called essential, if $\nu_{G,P}(\{\theta\})\neq 0$. 


\begin{lemma}\label{neutpoz}
For any vertex $u$, we have
\[\lim_{t\to 0}-\frac{it}{s_{G,u}(\theta+it)}=\sum_{v\sim u}\nu_{G-u,v}(\{\theta\}),\]
or equivalently,
\[\lim_{t\to 0}-\frac{s_{G,u}(\theta+it)}{it}=\left(\sum_{v\sim u}\nu_{G-u,v}(\{\theta\})\right)^{-1},\]
with the convention that $0^{-1}=\infty$.
\end{lemma}
\begin{proof}
Consider the identity in Lemma~\ref{rekurz} with $z=\theta+it$. Multiplying it $\frac{it}{s_{G,u}(\theta+it)}$, we obtain that
\[\frac{it}{s_{G,u}(\theta+it)}=it(\theta+it)-\sum_{v\sim u} it\cdot s_{G-u,v}(\theta+it).\]
Letting $t\to 0$ and applying Lemma~\ref{limit_it}, we get that statement.  
\end{proof}

\begin{definition}
A vertex $u$ is \emph{neutral} if it is non-essential and
\[\sum_{v\sim u}\nu_{G-u,v}(\{\theta\})=0.\]
A vertex $u$ is \emph{positive} if 
\[\sum_{v\sim u}\nu_{G-u,v}(\{\theta\})>0.\footnote{Note that in this case it follows that $u$ is non-essential.}\]

A vertex  is called \emph{special} if it is non-essential, but it has an essential neighbor.

\end{definition}

From Lemma~\ref{neutpoz}, we have the following lemma. 

\begin{lemma}\label{neutpoz2}\hfill
\begin{enumerate}[(i)]
\item A vertex $u$ is essential if and only if  \[\lim_{t\to 0} it\cdot s_{G,u}(\theta+it)>0.\]

\item A vertex $u$ is neutral if and only if it is non-essential and \[\lim_{t\to 0}-\frac{s_{G,u}(\theta+it)}{it}=\infty.\]

\item A vertex $u$ is positive if and only if  \[\lim_{t\to 0}-\frac{s_{G,u}(\theta+it)}{it}\] is finite and positive.\qed
\end{enumerate}
\end{lemma} 

The next lemma can be obtained by combining Lemma~\ref{neutpoz2} and Lemma~\ref{sGminthanyados}.
\begin{lemma}\label{finiteessneutpoz}
Let $G$ be a finite graph, and $u$ be a vertex of $G$. Then
\begin{enumerate}[(i)]
    \item The vertex $u$ is essential if and only if $\mult(\theta,G-u)=\mult(\theta,G)-1$.  
    \item The vertex $u$ is neutral if and only if $\mult(\theta,G-u)=\mult(\theta,G)$.
    \item The vertex $u$ is positive if and only if $\mult(\theta,G-u)=\mult(\theta,G)+1$.
\end{enumerate}\qed
\end{lemma}

\begin{lemma}
If $G$ is finite, then it has no $0$-neutral vertices.
\end{lemma}
\begin{proof}
Observe that $\mult(0,G)\equiv |V(G)|$ modulo $2$. Thus, the statement follows from Lemma~\ref{finiteessneutpoz}.
\end{proof} 

Note that the statement above is not true for infinite graphs. For example, let $G$ be a semi-infinite path, and let $o$ be its end vertex. Let $r$ be the unique neighbour of $o$. Then $(G,o)$ is isomorphic to $(G-o,r)$. In particular, the type of $o$ in $G$ is the same as the type of $r$ in $G-o$. It is only possible, if $o$ is $0$-neutral. We can even give a unimodular example. Indeed, in the bi-infinite path every vertex will be $0$-neutral. To see this, observe that it follows from the previous example that no vertex can be $0$-positive. Also, no vertex can be essential, because of Theorem~\ref{thmtransitive}.\footnote{There is another way to show that  no vertex is essential. Namely, we know that the matching measure is given by the Kesten-McKay measure~\cite{mckay}, which is absolutely continuous.} In fact, a similar argument shows that all the vertices of a $d$-regular infinite tree are $0$-neutral.  

\section{The effects of deleting vertices}\label{sec:delver}

\subsection{Stability results}\label{sec:stability}

The content of this section is summarised in Table~\ref{table:stab}.

\begin{table}[!h]
\begin{center}
\begin{tabular}{c|c||c || c}
     $u$ in $G$ & $a$ in $G$ & $u$ in $(G-a)$ & \\
     \hline
     essential & non-essential (*) & essential & Corollary~\ref{essstab}\\
     special & non-essential (*)& special  & Lemma~\ref{specstab}\\
     positive & special (*) & positive & Lemma~\ref{posstab}\\
     neutral & special & neutral & Lemma~\ref{neustab}\\
     positive & essential & positive & Lemma~\ref{esspos}\\
     neutral & essential & neutral & Lemma~\ref{esspos}\\
     non-essential & neutral & non-essential & Lemma~\ref{neu}\\
     neutral & positive & non-positive& Lemma~\ref{posneu}\\
     positive & positive & non-neutral & Lemma~\ref{pospos}
\end{tabular}
\vspace{4mm}
\caption{Stability results. Stars indicate the cases, when we can delete arbitrary many vertices of the given type.}
\label{table:stab}
\end{center}
\end{table}
Note that for finite graphs most of these results were proved in \cite{godsil,kucheng}.

Given a vertex $u$, let $\Pi_{G,u}$ be the orthogonal projection to the $\theta$-eigenspace of the adjacency operator of $T(G,u)$. It follows from the Spectral theorem that for a path $P=(p_0,p_1,\dots,p_k)$, we have $\nu_{G,P}(\{\theta\})=\langle \Pi_{G,p_0} \chi_{p_0},\chi_P\rangle$.
\begin{lemma}\label{essend}
If a path $P$ is essential, then both endpoints of $P$ are essential.
\end{lemma}
\begin{proof}
Assume that $P$ has a non-essential endpoint $u$. Then \[ \|\Pi_{G,u}\chi_u\|_2^2=\langle \Pi_{G,u}\chi_u,\chi_u\rangle=0,\] that is, $\Pi_{G,u}\chi_u=0$. In particular, $\langle \Pi_{G,u}\chi_u,\chi_P\rangle=0$, so $P$ is non-essential, which is contradiction.
\end{proof}

Given a subset $A$ of the vertices, and a vertex $u$. Let $\mathcal{P}(u,A)$ be the set of paths starting at $u$ and ending in $A$ without any inner vertex in $A$. 

\begin{lemma}\label{pathlemma}
Let $A$ be a subset of vertices, and let $u$ be a vertex not in $A$. Assume that each path in $\mathcal{P}(u,A)$ is non-essential. Then
\[\nu_{G-A,u}(\{\theta\})\ge \nu_{G,u}(\{\theta\}).\]

In particular, if $u$ is essential in $G$, then $u$ is essential in $G-A$.
\end{lemma}  
\begin{proof}
Clearly, we may assume that $u$ is essential in $G$, since the statement is trivial otherwise. Let $w$ be the projection of $\Pi_{G,u}\chi_u$ to the components in $\mathcal{P}_{G-A}(u)$. Since for each $P\in \mathcal{P}(u,A)$, we have $\langle \Pi_{G,u}\chi_u,\chi_P\rangle=0$, we see that $w$ is in the $\theta$-eigenspace of the adjacency operator of $T(G-A,u)$. Recall the elementary fact that
\[\langle \Pi_{G-A,u}\chi_u,\chi_u\rangle = \sup_h |\langle h,\chi_u\rangle|^2, \]
  where the supremum is over all $h$ such that $\|h\|_2=1$ and $h$ is in the $\theta$-eigenspace of the adjacency operator of $T(G-A,u)$. Thus,
\begin{align*}
\langle \Pi_{G-A,u}\chi_u,\chi_u\rangle&\ge |\langle \|w\|_2^{-1} w,\chi_u\rangle| ^2=\|w\|_2^{-2} |\langle w,\chi_u \rangle|^2=\|w\|_2^{-2} |\langle \Pi_{G,u}\chi_u,\chi_u \rangle|^2\\&=\frac{\|\Pi_{G,u}\chi_u\|^2}{\|w\|^2_2} \langle \Pi_{G,u}\chi_u,\chi_u \rangle\ge \langle \Pi_{G,u}\chi_u,\chi_u \rangle.
\end{align*}
\end{proof}

\begin{cor}\label{essstab}
Let $A$ be a subset of non-essential vertices, $u\not\in A$. Then
\[\nu_{G-A,u}(\{\theta\})\ge \nu_{G,u}(\{\theta\}).\]

In particular, if $u$ is essential in $G$, then $u$ is essential in $G-A$. 

\end{cor}
\begin{proof}
Note that all paths in $\mathcal{P}(u,A)$ are non-essential from Lemma~\ref{essend}. Thus, the previous lemma can be applied.
\end{proof}

Recall that a vertex  is called special if it is non-essential, but it has an essential neighbor.

\begin{lemma}
Let $u$ be an essential vertex, and let $w$ be a non-essential neighbor of $u$. Then $w$ is positive. In other words, every special vertex is positive.
\end{lemma}
\begin{proof}
From Corollary~\ref{essstab}, we know that the vertex $u$ is essential in $G-w$. So $w$ is positive from the definitions.
\end{proof}

\begin{lemma}\label{specstab}
Let $A$ be a subset of non-essential vertices, and let $u$ be a special vertex which is not in $A$. Then $u$ is special in $G-A$.
\end{lemma}
\begin{proof}
Let $w$ be an essential neighbor of $u$. We know that $u$ is non-essential. Therefore, from Corollary~\ref{essstab}, we have that $w$ is essential in $G-A-u$. This implies that $u$ is positive in $G-A$, in particular, $u$ is non-essential in $G-A$. From Corollary~\ref{essstab}, $w$ is essential in $G-A$.  Thus, $u$ is special in $G-A$.
\end{proof}

\begin{lemma}\label{posstab}
Let $A$ be a subset of the special vertices and let $u$ be a positive vertex not in $A$. Then $u$ is positive in $G-A$. 
\end{lemma}
\begin{proof}
Since $u$ is positive, there is a neighbor $w$ of $u$, which is essential in $G-u$. From Lemma~\ref{specstab}, we know that all the vertices of $A$ are special in the graph $G-u$. In particular, they are all non-essential. From Corollary~\ref{essstab}, $w$ is essential in $(G-u)-A$.  That is, $w$ is essential in $(G-A)-u$. Thus, $u$ is positive in $G-A$.  
\end{proof}

\begin{lemma}\label{stabnemess}
Let $a$ be a special vertex, and let $u$ be a non-essential vertex. Then $u$ is non-essential in $G-a$.  
\end{lemma}
\begin{proof}
From Lemma~\ref{specstab}, we see that $a$ is special in $G-u$. In particular, $a$ is positive in $G-u$. We have that
\[\lim_{t\to 0} -\frac{s_{G-u,a}(\theta+it)}{it}>0\] 
and finite, furthermore,
\[\lim_{t\to 0} it s_{G,u}(\theta+it)=0.\]
Therefore,
\[\lim_{t\to 0} -{s_{G-u,a}(\theta+it)\cdot s_{G,u}(\theta+it) }=0.\]
From Lemma~\ref{felcserel}, we have
\[s_{G-u,a}(z)\cdot s_{G,u}(z)=s_{G,a}(z)\cdot s_{G-a,u}(z).\]
Therefore,
\[\lim_{t\to 0} -{s_{G,a}(\theta+it)\cdot s_{G-a,u}(\theta+it) }=0.\]
Since
\[\lim_{t\to 0} -\frac{s_{G,a}(\theta+it)}{it}>0\] 
finite, we must have $\lim_{t\to 0} it s_{G-a,u}(\theta+it)=0$, that is, $u$ is not essential in $G-a$. 
\end{proof}

\begin{lemma}\label{neustab}
Let $a$ be a special vertex, and let $u$ be a neutral vertex. Then $u$ is neutral in $G-a$.  
\end{lemma}
\begin{proof}
From Lemma~\ref{specstab}, we see that $a$ is special in $G-u$. In particular, $a$ is positive in $G-u$. We have that
\[\lim_{t\to 0} -\frac{s_{G-u,a}(\theta+it)}{it}>0\] 
and finite, furthermore, 
\[\lim_{t\to 0} -\frac{s_{G,u}(\theta+it)}{it}=\infty.\]
Therefore,
\[\lim_{t\to 0} \frac{s_{G-u,a}(\theta+it)\cdot s_{G,u}(\theta+it) }{t^2}=\infty.\]
From Lemma~\ref{felcserel}, this implies that
\[\lim_{t\to 0} \frac{s_{G,a}(\theta+it)\cdot s_{G-a,u}(\theta+it) }{t^2}=\infty.\]
Since
\[\lim_{t\to 0} -\frac{s_{G,a}(\theta+it)}{it}>0\] 
finite, we must have $\lim_{t\to 0} -\frac{ s_{G-a,u}(\theta+it)}{it}=\infty$. Since we know from the previous lemma that $u$ is not essential in $G-a$, this implies that $u$ is neutral in $G-a$.  
\end{proof}

\begin{lemma}\label{esspos}
Let $a$ be an essential vertex and let $u$ be a positive (or neutral) vertex. Then $u$ is positive (or neutral) in $G-a$.
\end{lemma}
\begin{proof}
First we will show that if $u$ is non-essential in $G$, then $u$ is non-essential in $G-a$. To see this, observe
\[
    \lim_{t\to 0}it\cdot s_{G-u,a}(\theta+it)\cdot it\cdot s_{G,u}(\theta+it)=0,
\]
since $u$ is non-essential in $G$. 
On the other hand, by Lemma~\ref{felcserel}, this is the same as
\[
    \lim_{t\to 0}it\cdot s_{G-a,u}(\theta+it)\cdot it\cdot s_{G,a}(\theta+it)=0.
\]
Since $a$ is essential, this could only happen if 
\[
\lim_{t\to 0}it\cdot s_{G-a,u}(\theta+it)=0,
\]
that is, $u$ is non-essential in $G-a$.

To prove the statement of the lemma, observe that $a$ is essential in $G-u$ by Corollary~\ref{essstab}, therefore 
\[
\lim_{t\to 0} s_{G-u,a}(\theta+it)\cdot s_{G,u}(\theta+it)=\lim_{t\to 0} it\cdot s_{G-u,a}(\theta+it)\cdot \frac{s_{G,u}(\theta+it)}{it}
\]
is finite if $u$ is positive in $G$ (or infinite if $u$ is neutral in $G$).

By Lemma~\ref{felcserel}, we know that this is equal to 
\[
    \lim_{t\to 0} s_{G,a}(\theta+it)\cdot s_{G-a,u}(\theta+it)=\lim_{t\to 0} it\cdot s_{G,a}(\theta+it)\cdot \frac{s_{G-a,u}(\theta+it)}{it},
\]
where we know that $a$ is essential in $G$, thus
\[
    \lim_{t\to 0}-\frac{s_{G-a,u}(\theta+it)}{it}
\]
is  finite (or infinite), that is, $u$ is positive (or neutral) in $G-a$.
\end{proof}

\begin{lemma}\label{neu}
Let $a$ be a neutral vertex and $u$ be an other vertex of $G$. Then
\[
    \nu_{G,u}(\{\theta\})= \nu_{G-a,u}(\{\theta\}).
\]
In particular, $u$ is essential in $G$ if and only if $u$ is essential in $G-a$.
\end{lemma}
\begin{proof}
From  Corollary~\ref{essstab}, we have $\nu_{G-a,u}(\{\theta\})\ge \nu_{G,u}(\{\theta\})$. So it is enough to prove that $\nu_{G,u}(\{\theta\})\ge \nu_{G-a,u}(\{\theta\})$.
This is clear when $u$ is non-essential in $G-a$, 
so we will assume that $u$ is essential in $G-a$. Then $w=\Pi_{G-a,u}\chi_u\in \ell^2(\mathcal{P}_{G-a}(u))$ is a non-zero $\theta$-eigenvector  of the adjacency operator of $T(G-a,u)$. We would like to show that the natural extension $\hat w$ of $w$ with 0's as a vector of $\ell^2(\mathcal{P}_G(u))$ is a $\theta$-eigenvector of the adjacency operator of $T(G,u)$.

To see this, observe that any neighbor of $a$ is non-essential in $G-a$, thus for any path $P\in\mathcal{P}(u,a)$ the subpath $P'=P-a$ is non-essential by Lemma~\ref{essend}. This means that $\langle w,\chi_{P'}\rangle=0$, and therefore no eigenvalue-equation will fail if we extend $w$ with zeros.
 Thus, we have
\[
    \nu_{G,u}(\{\theta\})\ge \langle \|\hat w\|^{-1}\hat w,\chi_u \rangle ^2=\langle \|w\|^{-1}w,\chi_u \rangle ^2=\nu_{G-a,u}(\{\theta\}). \qedhere
\]
\end{proof} 

\begin{lemma}\label{posneu}
Let $a$ be positive vertex and $u$ be a neutral vertex. Then $u$ is non-positive in $G-a$.
\end{lemma}
\begin{proof}
For the sake of contradiction, assume that $u$ is positive in $G-a$. 
Then
\[
\lim_{t\to 0} \frac{s_{G,a}(\theta+it)}{-it}\frac{s_{G-a,u}(\theta+it)}{-it} >0
\]
is finite. On the other hand, by Lemma~\ref{felcserel}, this is equal to
\[
    \lim_{t\to 0} \frac{s_{G,u}(\theta+it)}{-it}\frac{s_{G-u,a}(\theta+it)}{-it}.
\]
Since $u$ is neutral, we see that $\lim_{t\to 0}\frac{s_{G,u}(\theta+it)}{-it}$ is infinite, thus $\lim_{t\to 0}\frac{s_{G-u,a}(\theta+it)}{-it}$ has to be 0 and that is impossible.
\end{proof}
\begin{samepage}
\begin{lemma}\label{pospos}
Let $a$ be positive vertex and $u$ be an other positive vertex. Then $u$ is non-neutral  in $G-a$.
\end{lemma}
\begin{proof}
For the sake of contradiction, assume that $u$ is neutral in $G-a$. So
\[
    \lim_{t\to 0} s_{G,a}(\theta+it)\cdot s_{G-a,u}(\theta+it)=\lim_{t\to 0}\frac{s_{G,a}(\theta+it)}{it}\cdot it\cdot s_{G-a,u}(\theta+it)=0.
\]
By definition, there is a neighbor $w$ of $u$, such that $w$ is essential in $G-u$. Since $u$ is neutral in $G-a$, therefore,  $w$ has to be non-essential in $G-a-u$. By Corollary~\ref{essstab}, this could happen only if $a$ is essential in $G-u$. But it means that
\[
    \lim_{t\to 0} s_{G,u}(\theta+it)\cdot s_{G-u,a}(\theta+it)=\lim_{t\to 0} \frac{s_{G,u}(\theta+it)}{it}\cdot it\cdot s_{G-u,a}(\theta+it)<0,
\]
which contradicts to Lemma~\ref{felcserel}.
\end{proof}
\end{samepage}

\subsection{A Christoffell-Darboux type formula}
\begin{samepage}
\begin{lemma}\label{lemmavaltozas}
Let $K$ be a subset of vertices, and let $u$ be a vertex not in $K$. For path $P\in \mathcal{P}(u,K)$, let $P'$ be the path obtained from $P$ by deleting the endpoint of $P$  in $K$. Then, for $z\in H$, we have
\[s_{G,u}(z)-s_{G-K,u}(z)=\sum_{P\in\mathcal{P}(u,K)} s_{G-K,P'}(z)s_{G,P}(z),\]
where the sum on the right converges absolutely.
\end{lemma}
\end{samepage}
\begin{proof}

Let $A$ be the adjacency operator of $T(G,u)$. Let $B$ be the adjacency operator of the subtree of $T(G,u)$ induced by $\mathcal{P}_{G-K}(u)\subseteq  \mathcal{P}_G(u)$. Note that this subtree can by identified with $T(G-K,u)$.

From the second resolvent identity, we get 
\begin{align*}
s_{G,u}(z)-s_{G-K,u}(z)&=\langle (zI-A)^{-1}\chi_u,\chi_u\rangle-\langle (zI-B)^{-1}\chi_u,\chi_u\rangle\\&=\langle (zI-A)^{-1}(A-B)(zI-B)^{-1}\chi_u,\chi_u\rangle\\
&=\langle (A-B)(zI-B)^{-1}\chi_u,(\overline{z}I-A)^{-1}\chi_u\rangle.
\end{align*}
Note that $(A-B)(zI-B)^{-1}\chi_u$ is supported on $\mathcal{P}(u,K)$. Moreover, for each \break $P\in \mathcal{P}(u,K)$, we have $\langle (A-B)(zI-B)^{-1}\chi_u,\chi_{P}\rangle=\langle (zI-B)^{-1}\chi_u,\chi_{P'}\rangle$. Thus,
\begin{align*}
\langle (A-B)(zI-B)^{-1}\chi_u,&(\overline{z}I-A)^{-1}\chi_u\rangle\\
&=\sum_{P\in\mathcal{P}(u,K)} \langle (zI-B)^{-1}\chi_u,\chi_{P'}\rangle \overline{\langle(\overline{z}I-A)^{-1}\chi_u,\chi_P\rangle}\\
&=\sum_{P\in\mathcal{P}(u,K)} \langle (zI-B)^{-1}\chi_u,\chi_{P'}\rangle \langle(zI-A)^{-1}\chi_u,\chi_P\rangle\\
&=\sum_{P\in\mathcal{P}(u,K)} s_{G-K,P'}(z)s_{G,P}(z).
\end{align*}
\end{proof}

The following convention will be useful for us.
\begin{conv}\label{conv1}
 If $u\in K$, we define $s_{G-K,u}(z)\equiv 0$. Moreover, we define $s_{G-K,P'}\equiv 1$ for the empty path $P'$. 
\end{conv}  
With these conventions, Lemma~\ref{lemmavaltozas} remains true even in the case of $u\in K$.
\begin{cor}\label{cornegyzetosszeg}
Let $u$ and $v$ be two vertices. Then
\[s_{G,u}(z)s_{G,v}(z)-s_{G,v}(z)s_{G-v,u}(z)=\sum_{P\in \mathcal{P}(u,v)} \left(s_{G,P}(z)\right)^2.\]
Note that if $u=v$, then this statement should be interpreted using Convention~\ref{conv1}.
\end{cor}  
\begin{proof}
We apply the previous lemma for $K=\{v\}$, and multiply that identity with~$s_{G,v}(z)$.
\end{proof}

Note that for finite graphs, Corollary~\ref{cornegyzetosszeg} and Lemma~\ref{lemmavaltozas} are special cases of the more general formula of Heilmann and Lieb \cite[Theorem 6.3]{heilmann}.

\subsection{The total change of the measure of an atom deleting a single vertex}

 Our aim in this subsection is to prove the following infinite analogue of Lemma~\ref{finiteessneutpoz}. 
\begin{lemma}\label{deltasum}\hfill

If $u$ is positive, then
\[\sum_{v\neq u} (\nu_{G-u,v}(\{\theta\})-
\nu_{G,v}(\{\theta\}))=1.\]

If $u$ is neutral, then
\[\sum_{v\neq u} (\nu_{G-u,v}(\{\theta\})-
\nu_{G,v}(\{\theta\}))=0.\]

If $u$ is essential, then
\[\sum_{v\in V(G)} (\nu_{G-u,v}(\{\theta\})-
\nu_{G,v}(\{\theta\}))=-1.\]
Here we use the convention that $\nu_{G-u,u}(\{\theta\})=0$.
\end{lemma}

First, we handle the case when $u$ is neutral. Observe that by Lemma~\ref{neu}  each term of the sum is $0$, so we have proved the second statement. In the rest of the subsection, we will focus on the cases when $u$ is not a neutral vertex.

\begin{lemma}\label{negyzetosszeg}
For any vertex $u$, we have
\[\lim_{t\to 0}-t^2\sum_{P\in\mathcal{P}(u)} \left(s_{G,P}(\theta+it)\right)^2=\nu_{G,u}(\{\theta\}).\]
\end{lemma}
\begin{proof}
Let $A$ be the adjacency operator of $T(G,u)$. Then
\begin{align*}
\lim_{t\to 0}-t^2\sum_{P\in\mathcal{P}(u)} \left(s_{G,P}(\theta+it)\right)^2&=
\lim_{t\to 0}-t^2\langle \left((\theta+it)I-A\right)^{-2}\chi_u,\chi_u\rangle\\&=\lim_{t\to 0} \int_{-D}^D \frac{-t^2}{(\theta-x+it)^2} d\nu_{G,u}(x)\\&=\int_{-D}^D \lim_{t\to 0}  \frac{-t^2}{(\theta-x+it)^2} d\nu_{G,u}(x)\\&=
\nu_{G,u}(\{\theta\}),
\end{align*}
where we can exchange the limit and the integral, because of the dominated convergence theorem. 
\end{proof}

\begin{lemma}
Let $V$ be a countable set. For all $t\ge 0$, we are given the vectors $x_t,y_t\in \ell^2(V)$, such that $\lim_{t\to 0} x_t=x_0$ and $\lim_{t\to 0} y_t=y_0$ in  $\ell^2$-norm. Let $x_t\circ y_t$ be the pointwise product of $x_t$ and $y_t$. Then for all $t\ge 0$, we have $x_t\circ y_t\in \ell^1(V)$. Moreover, $\lim_{t\to 0} x_t\circ y_t=x_0\circ y_0$ in $\ell^1$-norm.  
\end{lemma}
\begin{proof}
From the Cauchy-Schwarz-Bunyakovsky inequality, we have that \break $\|x_t\circ y_t\|_1\le\|x_t\|_2\|y_t\|_2<\infty$, so $x_t\circ y_t$ is indeed in $\ell^1(V)$.
 It follows from the convergence of $x_t$ and $y_t$, that there is a $K$ such that $\|x_t\|_2,\|y_t\|_2<K$ for every small enough $t$. Therefore, 
\begin{align*}
\|x_t\circ y_t-x_0\circ y_0\|_1&\le \|(x_t-x_0)\circ y_t\|_1+\|x_0\circ (y_t-y_0)\|_1 \\&\le \|x_t-x_0\|_2 \|y_t\|_2+\|x_0\|_2\|y_t-y_0\|_2\\&\le K(\|x_t-x_0\|_2+\|y_t-y_0\|_2)
\end{align*} 
for small enough $t$. The statement follows from the convergence of $x_t$ and $y_t$.  
\end{proof}

\begin{samepage}
\begin{lemma}\label{felcs11}
Fix a vertex $w$ of a graph $G$. 
For $t>0$, we define a vector $r_t\in \mathbb{C}^{\mathcal{P}(w)}$ by setting
\[r_t(P)=-t^2\left(s_{G,P}(\theta+it)\right)^2\]
for  all $P\in \mathcal{P}(w)$. Moreover, we define $r_0\in \mathbb{C}^{\mathcal{P}(w)}$ by setting
\[r_0(P)=\left(\langle\Pi_{G,w}\chi_w,\chi_P\rangle\right)^2.\]
for all $P\in \mathcal{P}(w)$. 

Then $r_t\in \ell^1(\mathcal{P}(w))$ for all $t\ge 0$. 

Moreover, $\lim_{t\to 0}r_t=r_0$ in $\ell^1(\mathcal{P}(w))$.
\end{lemma}
\end{samepage}
\begin{proof}
 Let $A$ be the adjacency operator of $T(G,w)$. Let $t>0$. Observe that
\[r_t=it((\theta+it)I-A)^{-1}\chi_w\circ it((\theta+it)I-A)^{-1}\chi_w,\] 
and \[r_0=\Pi_{G,w}\chi_w\circ \Pi_{G,w}\chi_w.\]

The statement will follow from the the previous lemma, once we prove that 
\[\lim_{t\to 0}\|it((\theta+it)I-A)^{-1}\chi_w-\Pi_{G,w}\chi_w\|_2=0.\]

Observe that 
\begin{align*}
&\left\|\left(it((\theta+it)I-A)^{-1}-\Pi_{G,w}\right)\chi_w\right\|^2_2\\&\qquad=
\left\langle\left(it((\theta+it)I-A)^{-1}-\Pi_{G,w}\right)^*\left(it((\theta+it)I-A)^{-1}-\Pi_{G,w}\right)\chi_w,\chi_w\right\rangle
\\&\qquad=\int_{-D}^D \left|\frac{it}{\theta+it-x}-\delta_{\theta,x}\right|^2 d\nu_{G,w}(x),
\end{align*}
and  $\left|\frac{it}{\theta+it-x}-\delta_{\theta,x}\right|^2\le 4$. Moreover, $\lim_{t\to 0} \left|\frac{it}{\theta+it-x}-\delta_{\theta,x}\right|^2=0$ for any fixed $x$. Thus, from the dominated convergence theorem, we obtain that 
\[\lim_{t\to 0}\left\|\left(it((\theta+it)I-A)^{-1}-\Pi_{G,w}\right)\chi_w\right\|^2_2=0. \qedhere\]
\end{proof}

\begin{lemma}\label{felcs2}
For any vertex $w$, we have
\begin{multline*}
\lim_{t\to 0} \sum_{v\in V(G)}-t^2\cdot s_{G,w}(\theta+it)\left(s_{G,v}(\theta+it)-s_{G-w,v}(\theta+it)\right)\\=\sum_{v\in V(G)}\lim_{t\to 0}-t^2\cdot s_{G,w}(\theta+it)\left(s_{G,v}(\theta+it)-s_{G-w,v}(\theta+it)\right).
\end{multline*}
Here we used Convention~\ref{conv1}. 
\end{lemma}
\begin{proof}
Consider the linear operator $\tau:\ell^1(\mathcal{P}(w))\to\ell^1(V(G))$ defined by setting 
\[(\tau x)(v)=\sum_{P\in \mathcal{P}(w,v)} x(P)\]
for all $x\in \ell^1(\mathcal{P}(w))$ and $v\in V(G)$.

From  Corollary~\ref{cornegyzetosszeg}, we have
\[(\tau r_t)(v)=-t^2\cdot s_{G,w}(\theta+it)\left(s_{G,v}(\theta+it)-s_{G-w,v}(\theta+it)\right)\]
for all $v\in V(G)$.

By the triangle inequality, $\tau$ is an operator of norm at most $1$. Combining this with Lemma~\ref{felcs11}, we have $\lim_{t\to 0}\tau r_t=\tau r_0$ in $\ell^1(V(G))$. Similarly, the map $y\mapsto \sum_{v\in V(G)} y(v)$ is a continuous map from $\ell^1(V(G))$ to $\mathbb{C}$. Thus,
\[\lim_{t\to 0} \sum_{v\in V(G)} (\tau r_t)(v)=\sum_{v\in V(G)}\lim_{t\to 0} (\tau r_t)(v).\] 
This is exactly the statement of the lemma.
\end{proof}

In the next four lemmas, we will use the following notation. For a path $P\in \mathcal{P}(u)$ we denote by $P'$ the path obtained from $P$ by deleting its endpoint $u$.  We will use Convention~\ref{conv1} several times without mentioning it.

In the next two lemmas, we will handle the case of Lemma~\ref{deltasum}, when $u$ is an essential vertex in $G$.

\begin{lemma}\label{csere-ess}
Let $u$ be an essential vertex of $G$. Then
\[\sum_{v\in V(G)} (\nu_{G,v}(\{\theta\})-
\nu_{G-u,v}(\{\theta\}))=
\lim_{t\to 0} it
 \sum_{P\in \mathcal{P}(u)} s_{G,P}(\theta+it)s_{G-u,P'}(\theta+it).\] 
\end{lemma}
\begin{proof}
We have
\begin{align*}
\lim_{t\to 0} it \sum_{P\in \mathcal{P}(u)} s_{G,P}(\theta+it)&s_{G-u,P'}(\theta+it)\\
&=\lim_{t\to 0} \frac {it}{ s_{G,u}(\theta+it)} \sum_{P\in \mathcal{P}(u)}\left( s_{G,P}(\theta+it)\right)^2\\
&=\lim_{t\to 0} \frac {1}{ it \cdot s_{G,u}(\theta+it)} \cdot\left(-t^2\sum_{P\in \mathcal{P}(u)}\left( s_{G,P}(\theta+it)\right)^2\right).
\end{align*}
By Corollary~\ref{cornegyzetosszeg} and Lemma~\ref{felcs2}, we have
\begin{align*} 
\lim_{t\to 0}-t^2\sum_{P\in \mathcal{P}(u)}&\left( s_{G,P}(\theta+it)\right)^2\\
&=\lim_{t\to 0}-t^2\sum_{v\in V(G)}\quad \sum_{P\in \mathcal{P}(u,v)}\left( s_{G,P}(\theta+it)\right)^2\\
&=\lim_{t\to 0} \sum_{v\in V(G)}-t^2\cdot s_{G,u}(\theta+it)\left(s_{G,v}(\theta+it)-s_{G-u,v}(\theta+it)\right)\\
&=\sum_{v\in V(G)}\lim_{t\to 0}-t^2\cdot s_{G,u}(\theta+it)\left(s_{G,v}(\theta+it)-s_{G-u,v}(\theta+it)\right).
\end{align*}
Moreover, $\lim_{t\to 0} \frac {1}{ it \cdot s_{G,u}(\theta+it)} $ exists and it is finite. Thus,
\begin{align*}
\lim_{t\to 0} \frac {1}{ it \cdot s_{G,u}(\theta+it)}& \cdot\left(-t^2\sum_{P\in \mathcal{P}(u)}\left( s_{G,P}(\theta+it)\right)^2\right)\\
&=\sum_{v\in V(G)}\lim_{t\to 0}it\left(s_{G,v}(\theta+it)-s_{G-u,v}(\theta+it)\right)\\&=\sum_{v\in V(G)}(\nu_{G,v}(\{\theta\})-\nu_{G-u,v}(\{\theta\})).
\end{align*}
\end{proof}
\begin{lemma}\label{lemmapm-ess}
Let $u$ be an essential vertex of $G$. Then
\[\lim_{t\to 0} it \sum_{P\in \mathcal{P}(u)} s_{G,P}(\theta+it)s_{G-u,P'}(\theta+it)=1.\] 
\end{lemma}
\begin{proof}
We have
\begin{align*}
\lim_{t\to 0} it \sum_{P\in \mathcal{P}(u)} s_{G,P}(\theta+it)&s_{G-u,P'}(\theta+it)=\lim_{t\to 0} \frac {it}{ s_{G,u}(\theta+it)} \sum_{P\in \mathcal{P}(u)}\left( s_{G,P}(\theta+it)\right)^2\\&=\lim_{t\to 0} \frac {1}{ it \cdot s_{G,u}(\theta+it)} \cdot\left(-t^2\sum_{P\in \mathcal{P}(u)}\left( s_{G,P}(\theta+it)\right)^2\right).
\end{align*}
Here, by Lemma~\ref{negyzetosszeg}, we have 
\[\lim_{t\to 0}-t^2\sum_{P\in \mathcal{P}(u)}\left( s_{G,P}(\theta+it)\right)^2=\nu_{G,u}(\{\theta\})\]
and
\[\lim_{t\to 0}  it \cdot s_{G,u}(\theta+it)=\nu_{G,u}(\{\theta\}),\]
so the statement follows.
\end{proof}

Combining Lemma~\ref{csere-ess} and Lemma~\ref{lemmapm-ess}, we get Lemma~\ref{deltasum} for essential vertices.

In the rest of this subsection, we prove the analogues of  Lemma~\ref{csere-ess} and \break Lemma~\ref{lemmapm-ess} for positive vertices. The statements are almost the same, but the proofs are slightly more involved. 
\begin{lemma}\label{csere-pos}
Let $u$ be a positive vertex of $G$. Then
\[\sum_{v\in V(G)} (\nu_{G,v}(\{\theta\})-
\nu_{G-u,v}(\{\theta\}))=
\lim_{t\to 0} it
 \sum_{P\in \mathcal{P}(u)} s_{G,P}(\theta+it)s_{G-u,P'}(\theta+it).\] 
\end{lemma}
\begin{proof}
 We have
\begin{align*}
\lim_{t\to 0}& it \sum_{P\in \mathcal{P}(u)} s_{G,P}(\theta+it)s_{G-u,P'}(\theta+it)\\
&=\lim_{t\to 0} it\cdot s_{G,u}(\theta+it)\left(1+\sum_{w\sim u}\sum_{P\in \mathcal{P}_{G-u}(w)} \left(s_{G-u,P}(\theta+it)\right)^2\right)\\
&=\lim_{t\to 0} it\cdot s_{G,u}(\theta+it)\left(\sum_{w\sim u}\sum_{P\in \mathcal{P}_{G-u}(w)} \left(s_{G-u,P}(\theta+it)\right)^2\right)\\
&=\lim_{t\to 0} \frac{s_{G,u}(\theta+it)}{it}\left(\sum_{w\sim u}-t^2\sum_{P\in \mathcal{P}_{G-u}(w)} \left(s_{G-u,P}(\theta+it)\right)^2\right)\\
&=\lim_{t\to 0} \frac{s_{G,u}(\theta+it)}{it}\left(\sum_{w\sim u}-t^2\sum_{v\in V(G)-u}\quad \sum_{P\in \mathcal{P}_{G-u}(w,v)} \left(s_{G-u,P}(\theta+it)\right)^2\right)\\
&=\lim_{t\to 0} \frac{s_{G,u}(\theta+it)}{it}\\
&\qquad \left(\sum_{w\sim u}\sum_{v\in V(G)-u} -t^2\cdot s_{G-u,w}(\theta+it) \left(s_{G-u,v}(\theta+it)-s_{G-u-w,v}(\theta+it)\right) \right).
\end{align*}
Here, for any neighbor $w$ of $u$, from Lemma~\ref{felcs2}, we have
\begin{align*}
\lim_{t\to 0}& \sum_{v\in V(G)-u}-t^2\cdot s_{G-u,w}(\theta+it)\left(s_{G-u,v}(\theta+it)-s_{G-u-w,v}(\theta+it)\right)\\&=\sum_{v\in V(G)-u}\lim_{t\to 0}-t^2\cdot s_{G-u,w}(\theta+it)\left(s_{G-u,v}(\theta+it)-s_{G-u-w,v}(\theta+it)\right).
\end{align*}
Since $u$ is positive, $\lim_{t\to 0} \frac{s_{G,u}(\theta+it)}{it}$ exists and finite. Thus,
\begin{align*}
\lim_{t\to 0}&\frac{s_{G,u}(\theta+it)}{it} \sum_{v\in V(G)-u}-t^2\cdot s_{G-u,w}(\theta+it)\left(s_{G-u,v}(\theta+it)-s_{G-u-w,v}(\theta+it)\right)\\
&=\sum_{v\in V(G)-u}\lim_{t\to 0}-it s_{G,u}(\theta+it)\cdot s_{G-u,w}(\theta+it)\left(s_{G-u,v}(\theta+it)-s_{G-u-w,v}(\theta+it)\right).
\end{align*}
Therefore,
\begin{multline*}
\lim_{t\to 0} it \sum_{P\in \mathcal{P}(u)} s_{G,P}(\theta+it)s_{G-u,P'}(\theta+it)\\
=\sum_{v\in V(G)-u}\lim_{t\to 0}\sum_{w\sim u}-it s_{G,u}(\theta+it)\cdot s_{G-u,w}(\theta+it)\left(s_{G-u,v}(\theta+it)-s_{G-u-w,v}(\theta+it)\right).
\end{multline*}
Now fix a $v\neq u$, then
\begin{align*}
\lim_{t\to 0}&\sum_{w\sim u} it s_{G,u}(\theta+it)\cdot s_{G-u,w}(\theta+it)\left(s_{G-u,v}(\theta+it)-s_{G-u-w,v}(\theta+it)\right)\\
&=\lim_{t\to 0}\big(it \cdot s_{G-u,v}(\theta+it)\sum_{w\sim u} s_{G,(u,w)}(\theta+it)
\\& \qquad\qquad-it \cdot s_{G,v}(\theta+it)\sum_{w\sim u} s_{G-v,(u,w)}(\theta+it)\big)\\
&=\lim_{t\to 0}\big(it \cdot s_{G-u,v}(\theta+it)((\theta+it)s_{G,u}(\theta+it)-1)\\
&\qquad\qquad -it \cdot s_{G,v}(\theta+it)((\theta+it)s_{G-v,u}(\theta+it)-1)\big)\\
&=\lim_{t\to 0} \left(it\cdot s_{G,v}(\theta+it)-it\cdot s_{G-u,v}(\theta+it)\right)\\
&=\nu_{G,v}(\{\theta\})-\nu_{G-u,v}(\{\theta\}),
\end{align*}
where we used Lemma \ref{rekurz} at the second equality.
\end{proof}

\begin{lemma}\label{lemmapm-pos}
Let $u$ be a positive vertex of $G$. Then
\[\lim_{t\to 0} it \sum_{P\in \mathcal{P}(u)} s_{G,P}(\theta+it)s_{G-u,P'}(\theta+it)=-1.\] 
\end{lemma}
\begin{proof}
We have
\begin{align*}
\lim_{t\to 0}& ~it \sum_{P\in \mathcal{P}(u)} s_{G,P}(\theta+it)s_{G-u,P'}(\theta+it)\\
&=\lim_{t\to 0} it\cdot s_{G,u}(\theta+it)\left(1+\sum_{w\sim u}\sum_{P\in \mathcal{P}_{G-u}(w)} \left(s_{G-u,P}(\theta+it)\right)^2\right)\\
&=\lim_{t\to 0} it\cdot s_{G,u}(\theta+it)\left(\sum_{w\sim u}\sum_{P\in \mathcal{P}_{G-u}(w)} \left(s_{G-u,P}(\theta+it)\right)^2\right)\\
&=\lim_{t\to 0} \frac{ s_{G,u}(\theta+it)}{it}\left(-t^2\sum_{w\sim u}\sum_{P\in \mathcal{P}_{G-u}(w)} \left(s_{G-u,P}(\theta+it)\right)^2\right).
\end{align*}

Here
\[\lim_{t\to 0} -t^2\sum_{w\sim u}\sum_{P\in \mathcal{P}_{G-u}(w)} \left(s_{G-u,P}(\theta+it)\right)^2=\sum_{w\sim u} \nu_{G-u,w}(\{\theta\})>0,\]
from Lemma~\ref{negyzetosszeg}, and
\[\lim_{t\to 0}\frac{s_{G,u}(\theta+it)}{it}=-\left(\sum_{v\sim u}\nu_{G-u,v}(\{\theta\})\right)^{-1}\]
from Lemma~\ref{neutpoz}. So the statement follows.
\end{proof}

Combining Lemma~\ref{csere-pos} and Lemma~\ref{lemmapm-pos}, we get Lemma~\ref{deltasum} for positive vertices.

\subsection{Deleting several vertices}

Given a graph $G$, a subset $U$ of its vertices, and a vertex $v\not\in U$, we define
\[\Delta_{G,U}(v)=\nu_{G-U,v}(\{\theta\})-\nu_{G,v}(\{\theta\}).\]

\begin{lemma}\label{nneg1}
If all the vertices of $U$ are non-essential, then
\[\Delta_{G,U}(v)\ge 0.\]

\end{lemma}
\begin{proof}
This is a direct consequence of Corollary~\ref{essstab}. 
\end{proof}

Let $\ell:V\to [0,1]$ be a labeling of the vertices such that the labels are pairwise distinct. For a vertex $u\in U$, we define 
\[L(u)=\{w\in U|\ell(w)<\ell(u)\}.\footnote{This depends on the choice of $U$ and $\ell$, but we do not indicate this in the notation.}\]

Assume that all the vertices of $U$ are special. Then we define
\[\Delta_{G,U}^\ell(v)=\sum_{u\in U}\Delta_{G-L(u),\{u\}}(v).\]
The sum above is well defined, because all the terms are non-negative as the following lemma shows.
\begin{lemma}\label{nneg2}
Assume that all the vertices of $U$ are special. Let $U_0\subseteq  U$. Then for any subset $U_1\subseteq   U-U_0$, we have 
\[\Delta_{G-U_0,U_1}(v)\ge 0.\]
\end{lemma}
\begin{proof}
From Lemma~\ref{specstab}, we see that all the vertices of $U_1$ are special in $G-U_0$. Therefore, Lemma~\ref{nneg1} can be applied to get the statement. 
\end{proof}

\begin{lemma}\label{randomorder}
Assume that all the vertices of $U$ are special. Then
\[\Delta_{G,U}(v)\ge \Delta_{G,U}^\ell(v).\]
\end{lemma}
\begin{proof}
It is enough to prove that for any finite subset $F$ of $U$, we have
\[\Delta_{G,U}(v)\ge\sum_{u\in F}\Delta_{G-L(u),\{u\}}(v).\]

Let $F=\{u_1,u_2,\dots,u_m\}$, and assume that $\ell(u_1)<\ell(u_2)<\cdots<\ell(u_m)$. We set $U_0=\emptyset$, and for $i=1,2,\dots,m$, we set
\[U_{2i-1}=L(u_i),\]
\[U_{2i}=L(u_i)\cup\{u_i\}.\]
Finally, we set $U_{2m+1}=U$.
Observe that $U_0\subseteq  U_1\subseteq \dots\subseteq  U_{2m+1}$. So it is clear from the definitions, that
\[\Delta_{G,U}(v)=\sum_{i=0}^{2m} \Delta_{G-U_i,U_{i+1}-U_i}(v).\]
From Lemma~\ref{nneg2}, all the terms in the sum above are non-negative, thus,
\begin{align*}
\Delta_{G,U}(v)&=\sum_{i=0}^{2m} \Delta_{G-U_i,U_{i+1}-U_i}(v)\\&\ge \sum_{i=1}^{m} \Delta_{G-U_{2i-1},U_{2i}-U_{2i-1}}(v)\\&=\sum_{u\in F}\Delta_{G-L(u),\{u\}}(v). 
\end{align*}

\end{proof}

\section{Critical graphs -- Gallai's lemma and the  Gallai--Edmonds decomposition theorem}\label{sec:GGE}

In this section, we prove  Gallai's lemma and the Gallai-Edmonds decomposition theorem stated in the Introduction.

First, we would like to understand the connected components of the essential vertices. Recall that we say that a graph $G$ is critical if every vertex of $G$ is essential. 

\begin{lemma}\label{lemma:gallai}
Let $G$ be a connected critical graph and  let $u$ be a vertex of $G$. Then all the vertices of $G-u$ are non-essential.
\end{lemma}
\begin{proof}
Let $u$ be a vertex of $G$. For the sake of contradiction, assume that $G-u$ has an essential vertex. Since $u$ is not positive, all the neighbors of $u$ are non-essential in $G-u$. In particular, each connected component of $G-u$ contains a non-essential vertex. Since $G-u$ has an essential vertex, there must be a special vertex $w$ in $G-u$. From Lemma~\ref{stabnemess}, we have that all the neighbors of $u$ are non-essential in $G-u-w$, because they are non-essential in $G-u$.\footnote{It might happen that $w$ is a neighbor of $u$, in this case this statement is about the neighbors of $u$ other than $w$.}  We have $s_{G,u}(z)s_{G-u,w}(z)=s_{G,w}(z)s_{G-w,u}(z)$. Since $u$ is essential in $G$, $w$ is positive  in $G-u$, we have
\[\lim_{t\to 0} s_{G,u}(\theta+it)s_{G-u,w}(\theta+it)=\lim_{t\to 0} (it\cdot s_{G,u}(\theta+it))\frac{s_{G-u,w}(\theta+it)}{it}\]
is finite and non-zero.  Thus, $\lim_{t\to 0} s_{G,w}(\theta+it)s_{G-w,u}(\theta+it)$ is also finite and non-zero. Since $w$ is essential in $G$, this implies that $u$ is positive in $G-w$. Then, $u$ has a neighbor $x$ in $G-w$ such that $x$ is essential in $G-w-u$. This is a contradiction. 
\end{proof}

\begin{replemma}{egyszerugyok}
Let $G$ be a connected (possibly infinite) $\theta$-critical graph. Then
\[\sum_{u\in V(G)} \nu_{G,u}(\{\theta\})=1.\]
\end{replemma}  
\begin{proof}
Let $u$ be any vertex of $G$. From the previous lemma, all the vertices of $G-u$ are non-essential. Since $u$ is essential in $G$, from Lemma~\ref{deltasum}, we have
\[\sum_{v\in V(G)} \nu_{G,v}(\{\theta\})=\sum_{v\in V(G)} \left(\nu_{G,v}(\{\theta\})-\nu_{G-u,v}(\{\theta\})\right)=1.\qedhere\]
\end{proof}
As an easy corollary, we obtain Theorem~\ref{thmtransitive}.
\begin{reptheorem}{thmtransitive}
Let $G$ be an infinite connected vertex-transitive graph, let $o$ be any vertex of it. Then $\nu_{G,o}$ has no atoms, that is,
\[\nu_{G,o}(\{\theta\})=0\]
for any $\theta\in [-D,D]$.
\end{reptheorem}
\begin{proof}
Since $G$ is vertex-transitive, we have  $\nu_{G,u}(\{\theta\})=\nu_{G,v}(\{\theta\})$ for any two \break vertices $u,v$ of $G$. If all the  vertices were essential, then we would have \break $\sum_{u\in V(G)}\nu_{G,u}(\{\theta\})=\infty$, which contradicts Lemma~\ref{egyszerugyok}.
\end{proof}
\begin{cor}\label{cor:essdel}
If $u\neq v$ are in the same essential component $D'$ of $G$, then $v$ is non-essential in $G-u$.
\end{cor}
\begin{proof}
For the sake of contradiction, assume that $v$ is essential in $G-u$. All the vertices of $\partial D'$ are special in $G$, in particular, they are positive in $G$. Thus, by Lemma~\ref{esspos}, we know that all the vertices in $\partial D'$ are  positive in the graph $G-u$. By Corollary~\ref{essstab}, we know that $v$ is essential in $(G-u)-\partial D'$. On the other hand, we know that $G[D']$ is a critical graph by Corollary~\ref{essstab}. Thus, by Lemma~\ref{lemma:gallai}, any vertex in $G[D']-u$ has to be non-essential, which is a contradiction.
\end{proof}



In the rest of this section, we  prove the statements of Theorem~\ref{thm:GE_theta}.

\begin{lemma}
Let $G$ be a graph (possibly infinite), let $D$ be the set of $\theta$-essential vertices and $A=\partial D$. Then each component of $G[D]$ is critical.
\end{lemma}
\begin{proof}
By Corollary~\ref{essstab}, we know that every vertex in $(G-A)[D]=G[D]$ is essential, since every vertex in $A$ is positive. Since in every connected component of $G[D]$ all the vertices are essential, every component is critical.
\end{proof}

\begin{lemma}
Let $D$ be the set of $\theta$-essential vertices in $G$, and let $A=\partial D$ and let $X$ be a finite subset of $A$. Then there are at least $|X|+1$ connected components in $G[D]$ which are connected to a vertex in $X$ in the graph $G$.
\end{lemma}
\begin{proof}
Let $D'$ be the union of those connected components of $G[D]$ which are connected to a vertex of $X$ in $G$. 
Let $Y$ be  $\partial (X\cup D')$. 
Since all essential neighbors of $X$ are in $D'$, therefore every vertex in $Y$ is non-essential. Thus, by Corollary~\ref{essstab}, any vertex of $D'$ is essential in $G-Y$ and by Lemma~\ref{specstab}, any vertex of $X$ is special in $G-Y$. 

Let $G'$ be the subgraph of $G$ induced by $X\cup D'$. Observe that $G'$ is the union of connected components of $G-Y$. In particular, every vertex in $D'$ is essential in $G'$, and every vertex in $X$ is special in $G'$.  

We delete the vertices of $X$ one by one. By Lemma~\ref{specstab}, each vertex of $X$ is special at the moment of its deletion. 
By Lemma~\ref{deltasum}, each deleted vertex  decreases  the total weight of the atom $\theta$  by $1$.  After deleting all the vertices of $X$,  the resulting graph will have critical connected components by Corollary~\ref{essstab}.  Thus, by Lemma~\ref{egyszerugyok},  the total weight of the atom $\theta$ in $G'-X$ is the number of connected components $c(G'[D'])$ of $G'[D']$. Therefore, 
\[
    0<\sum_{u\in G'}\nu_{G',u}(\{\theta\})=\sum_{u\in G'-X}\nu_{G'-X,u}(\{\theta\})-|X|=c(G'[D'])-|X|.\qedhere
\]
\end{proof}

\section{The Gallai-Edmonds decomposition and the monomer-dimer model}\label{sec::monomerdimer}
In this section, we prove the statements of Theorem~\ref{thm:GE_0} part \eqref{GallaiEdmondsc}.
Part~\eqref{GallaiEdmondsa} and part~\eqref{GallaiEdmondsb} of Theorem~\ref{thm:GE_0} are covered by Theorem~\ref{thm:GE_theta}. The proof of part \eqref{GallaiEdmondsd} is postponed to Section~\ref{sec:uni3}.  


\begin{lemma}
For a  graph $G$ let $D$ be the set of $0$-essential vertices. Let $A=\partial D$ and let $C=V(G)-D-A$. Let $\mathcal{M}$ be   a Boltzmann random matching of $G$ at temperature zero. Then, the followings hold with probability~$1$:
\begin{enumerate}
    \item Every vertex in $A\cup C$ is covered by $\mathcal{M}$.
    \item Every vertex in $A$ is matched with a vertex in $D$.
    \item Every connected component of $G[D]$ contains at most one vertex not covered by $\mathcal{M}$.
    \item Every connected component of $G[D]$ contains at most one vertex which is matched with a  vertex in $A$.
\end{enumerate}
\end{lemma}
\begin{proof}
Let $t_1,t_2,\dots$ be a sequence of positive numbers tending to zero, such that $\mathcal{M}^{t_n}_{G}$ converges in law to $\mathcal{M}$. 

Let us prove the statements in order.
\begin{enumerate}
\item If $u\in A\cup C$, then $u$ is non-essential, thus
\[
    0=\nu_{G,u}(\{0\})=\lim_{n\to \infty} it_n \cdot s_{G,u}(it_n)=\mathbb{P}\left(u\textrm{ is uncovered by $\mathcal{M}$}\right).
\]

\item Let $u\in A$, and let $v$  be a neighbor of $u$, which is not in $D$. Note that $u$ is special in $G$, and $v$ is non-essential in $G$. From Lemma ~\ref{stabnemess}, we see that $v$ is non-essential in $G-u$. So we obtained that $\lim_{n\to \infty}\frac{s_{G,u}(it_n)}{it_n}$ is finite and $\lim_{n\to \infty} it_n s_{G-u,v}(it_n)$ is zero. Therefore, 
 \begin{align*}
    \mathbb{P}\left((u,v)\in\mathcal{M}\right)&=\lim_{n\to \infty } -s_{G,u}(it_n)s_{G-u,v}(it_n)\\
    &=\lim_{n\to \infty} -\frac{s_{G,u}(it_n)}{it_n}\cdot it_n s_{G-u,v}(it_n)\\
    &=\lim_{n\to \infty} -\frac{s_{G,u}(it_n)}{it_n}\cdot it_n s_{G-u,v}(it_n)=0.
 \end{align*}


\item Let $u$ and $v$ be two vertices of the same connected component $D'$ of $G[D]$. By Corollary~\ref{cor:essdel}, we have
 \begin{align*}
    \mathbb{P}\big(u,v \notin V(\mathcal{M})\big)&=\lim_{n\to \infty} (it_n)^2 s_{G,\{u,v\}}(it_n)\\&=\lim_{n\to \infty} it_n \cdot s_{G,u}(it_n)\cdot it_n \cdot s_{G-u,v}(it_n)
    \\&=\nu_{G,u}(\{0\})\cdot \nu_{G-u,v}(\{0\})=0.
 \end{align*}
 
 \item Let $e=(u,a_1)$ and $f=(v,a_2)$ be two vertex disjoint edges, where $u$ and $v$ are from the same connected component $D'$ of $G[D]$ and $a_1,a_2\in A$. Similarly as in the second part, we have
 \begin{align*}
    \qquad\quad\mathbb{P}\big(&e,f \in \mathcal{M} \big) \\
    &=\lim_{n\to \infty} s_{G,\{a_1,a_2,u,v\}}(it_n)\\
    &=\lim_{n\to \infty} s_{G,a_1}(it_n)s_{G-a_1,a_2}(it_n)s_{G-a_1-a_2,u}(it_n)s_{G-a_1-a_2-u,v}(it_n)\\
    &=\lim_{n\to \infty} \frac{s_{G,a_1}(it_n)}{-it_n}\frac{s_{G-a_1,a_2}(it_n)}{-it_n} \cdot it_n s_{G-a_1-a_2,u}(it_n)\cdot it_n s_{G-a_1-a_2-u,v}(it_n)\\
    &=\lim_{n\to \infty} \frac{s_{G,a_1}(it_n)}{-it_n}\frac{s_{G-a_1,a_2}(it_n)}{-it_n} \cdot it_n\cdot s_{G-a_1-a_2,u}(it_n)\cdot it\cdot s_{G-a_1-a_2-u,v}(it_n)\\
    &=0,
 \end{align*}
 because $a_1$ is special in $G$, $a_2$ is special in $G-a_1$, and $v$ is non-essential in $G-a_1-a_2-u$. Indeed, $a_2$ is special in $G-a_1$ by Lemma~\ref{specstab}. Moreover, $u$ and $v$ are in the same essential component of $G-a_1-a_2$ by Corollary~\ref{essstab}. So    Corollary~\ref{cor:essdel} can be applied to deduce that $v$ is non-essential in the graph $G-a_1-a_2-u$.\qedhere
\end{enumerate}
\end{proof}

\section{Unimodular graphs}\label{sec:unimodsec}

\subsection{The unimodular version of Gallai's lemma}

\begin{reptheorem}{lemmafinite}
Let $(G,o)$ be a unimodular random rooted graph. If $G$ is $\theta$-critical with probability $1$, then $G$ is finite with probability $1$.
\end{reptheorem}
\begin{proof}
Mark a vertex $u$ of $G$, if $\nu_{G,u}(\{\theta\})>\frac{1}2 \sup_{v\in V(G)} \nu_{G,v}(\{\theta\})$. 

Recall that by Theorem~\ref{egyszerugyok}, we have
\[\sum_{v\in V(G)} \nu_{G,v}(\{\theta\})=1.\]
From this, it is easy to see that we marked a non-empty finite subset of the vertices in a unimodular way. By Lemma~\ref{asfinite}, this is only possible if $G$ is finite almost surely.
\end{proof}

Let us denote by $\mathbb{A}\subset  \mathbb{R}$ the set of \emph{totally real algebraic integers}, that is, a real number $\theta$ is in $\mathbb{A}$ if $\theta$ is a root of a real-rooted monic polynomials with integer coefficients. 

\begin{cor}
Let $(G,o)$ be a unimodular random rooted graph, then any atom $\theta$ of the expected matching measure $\mathbb{E}\nu_{G,o}$ is a totally real  algebraic integer.
\end{cor}
\begin{proof}
Let $D$ be the set of $\theta$-essential vertices of $G$. 
Since $\theta$ is an atom, we have $\mathbb{P}(o\in D)>0$. For $o\in D$, let $\mathcal{C}_o$ be the connected component of $o$ in $G[D]$. Let $(G',o')$ have the same law as $(G,o)$ conditioned on the event $o\in D$. Then $(\mathcal{C}_{o'},o')$ is unimodular from Lemma~\ref{perc}. From Corollary~\ref{essstab}, we see that  $\mathcal{C}_{o'}$ has only essential vertices. Then Theorem~\ref{lemmafinite} shows that $\mathcal{C}_o$ is finite with probability  $1$. In particular, $\theta$ is a root of the matching polynomial of a finite graph.
\end{proof}


\subsection{An inequality}

\begin{theorem}\label{foegyenlotlenseg}
Let $(G,o)$ be a unimodular random rooted graph. Let $N$ and $S$ be disjoint unimodular subsets of the vertices, such that all the vertices in $N$ are non-essential, and all the vertices in $S$ are special.  Then
\[\mathbb{E}\mathbbm{1}(o\not\in N\cup S)\nu_{G-N-S,o}(\{\theta\})\ge \mathbb{P}(o\in S)+\mathbb{E}\nu_{G,o}(\{\theta\}).\]
\end{theorem}
\begin{samepage}
\begin{proof}
Note that, if $o\not\in N$, we have \[\nu_{G-N,o}(\{\theta\})\ge \nu_{G,o}(\{\theta\})\]
from Corollary~\ref{essstab}. Thus,
\[\mathbb{E}\nu_{G,o}(\{\theta\})=\mathbb{E}\mathbbm{1}(o\not\in N)\nu_{G,o}(\{\theta\})\le \mathbb{E} \mathbbm{1}(o\not\in N)\nu_{G-N,o}(\{\theta\}).\]
From Lemma~\ref{specstab}, all the vertices of $S$ are special in $G-N$. Then
\begin{align*}
\mathbb{E}\mathbbm{1}(o\not\in N\cup S )&\nu_{G-N-S,o}(\{\theta\})-\mathbb{E}\mathbbm{1}(o\not\in N)\nu_{G-N,o}(\{\theta\})\\&=\mathbb{E}\mathbbm{1}(o\not\in N\cup S)\left(\nu_{G-N-S,o}(\{\theta\})-\nu_{G-N,o}(\{\theta\})\right)\\&=\mathbb{E}\mathbbm{1}(o\not\in N\cup S)\Delta_{G-N,S}(o).
\end{align*}
So it is enough to prove that
\[\mathbb{E}\mathbbm{1}(o\not\in N\cup S)\Delta_{G-N,S}(o)\ge \mathbb{P}(o\in S).\]
Let $\ell$ be an i.i.d. uniform $[0,1]$ labeling of the vertices of $G-N$. It is clear that $(G,N,S,\ell,o)$ is unimodular, see \cite[Section 6]{aldous2007processes}. From Lemma~\ref{randomorder}, it is enough to prove the following lemma.
\begin{lemma}
We have
\[\mathbb{E}\mathbbm{1}(o\not\in N\cup S)\Delta_{G-N,S}^\ell(o)= \mathbb{P}(o\in S).\]
\end{lemma}
\begin{proof}
First, let 
\[f(G,N,S,\ell,x,y)=\mathbbm{1}(y\not\in N\cup S,x\in S)\cdot\Delta_{G-N-L(x),\{x\}}(y).\]
Observe that 
\[\mathbb{E}\sum_{v\in V(G)} f(G,N,S,\ell,v,o)=\mathbb{E}\mathbbm{1}(o\not\in N\cup S) \Delta_{G-N,S}^\ell(o).\]

Note that by Lemma~\ref{specstab} any vertex $v\in S-(L(o)\cup\{o\})$ is non-essential in both $G-N-L(o)$ and $G-N-(L(o)\cup \{o\})$. Thus, for any $v\in S-(L(o)\cup\{o\})$, we have $\Delta_{G-N-L(o),\{o\}}(v)=0$. Therefore,
\begin{align*}
\mathbb{E}\sum_{v\in V(G)} f(G,N,S,\ell,o,v)&=\mathbb{E}\mathbbm{1}(o\in S)\sum_{v\not\in N\cup S}\Delta_{G-N-L(o),\{o\}}(v)\\&=\mathbb{E}\mathbbm{1}(o\in S)\sum_{v\not\in N\cup L(o)\cup\{o\}}\Delta_{G-N-L(o),\{o\}}(v).
\end{align*}
So from the Mass-Transport Principle, we have
\begin{equation*}
\mathbb{E}\mathbbm{1}(o\not\in N\cup S) \Delta_{G-N,S}^\ell(o)=\mathbb{E}\mathbbm{1}(o\in S)\sum_{v\not\in N\cup L(o)\cup\{o\}}\Delta_{G-N-L(o),\{o\}}(v).
\end{equation*}
From Lemma~\ref{specstab}, we see that if $o\in S$, then $o$ is positive in $G-N-L(o)$, thus we can apply Lemma~\ref{deltasum} to obtain that 
\[\sum_{v\not\in N\cup L(o)\cup\{o\}}\Delta_{G-N-L(o),\{o\}}(v)=1.\] 
Therefore,
\begin{align*}\mathbb{E}\mathbbm{1}(o\not\in N\cup S) \Delta_{G-N,S}^\ell(o)&=\mathbb{E}\mathbbm{1}(o\in S)\sum_{v\not\in N\cup L(o)\cup\{o\}}\Delta_{G-N-L(o),\{o\}}(v)\\
&=\mathbb{P}(o\in S).\qedhere\end{align*}
\end{proof}
This concludes the proof of Theorem~\ref{foegyenlotlenseg}.
\end{proof}
\end{samepage}

\subsection{The proof of Theorem~\ref{thm25} and Theorem~\ref{thm2}}\label{sec:uni3}
\begin{reptheorem}{thm2}
Let $(G,o)$ be a unimodular random rooted graph. Let $\mathfrak{S}$ be the set of essential vertices of $G$. 
For $o\in \mathfrak{S}$, let $\mathcal{C}_o$ be the connected component of $o$ in  the induced subgraph $G[\mathfrak{S}]$. Then $\mathcal{C}_o$ is finite with probability $1$, and
\[\mathbb{E}\nu_{G,o}(\{\theta\})\le \mathbb{E}\mathbbm{1}(o\in \mathfrak{S}) |\mathcal{C}_o|^{-1}-\mathbb{P}(o\in \partial \mathfrak{S}).\] 
Moreover, for $\theta=0$, we have an equality in the line above. 
\end{reptheorem}
\begin{proof}
Applying Theorem~\ref{foegyenlotlenseg} with $S=\partial \mathfrak{S}$ and $N=V(G)-\mathfrak{S}-\partial\mathfrak{S}$, we get that
\[\mathbb{E}\nu_{G,o}(\{\theta\})\le \mathbb{E}\mathbbm{1}(o\in \mathfrak{S}) \nu_{G[\mathfrak{S}],o}(\{\theta\})-\mathbb{P}(o\in \partial \mathfrak{S}).\] 
Moreover, from Corollary~\ref{essstab}, we have that $G[\mathfrak{S}]$ has only essential vertices. 
Let $(\bar{G},\bar{o})$ have the same distribution as $(G,o)$ conditioned on the event that $o\in  \mathfrak{S}$. Then $(\mathcal{C}_{\bar{o}},\bar{o})$ is unimodular by Lemma~\ref{perc}. All the vertices of $\mathcal{C}_{\bar{o}}$  are essential. Thus, it follows from Theorem~\ref{lemmafinite} that $\mathcal{C}_o$ is finite with probability~$1$. 

Let us define
\[f(G,x,y)=\mathbbm{1}(x\in \mathfrak{S},y\in \mathcal{C}_x )|\mathcal{C}_x|^{-1}\nu_{G[\mathfrak{S}],x}(\{\theta\}).\]


It is clear that
\[\mathbb{E}\sum_{v\in G}f(G,o,v)=\mathbb{E} \nu_{G[\mathfrak{S}],o}(\{\theta\}).\] 
Form Lemma~\ref{egyszerugyok}, we have
\[\mathbb{E}\sum_{v\in G}f(G,v,o)=\mathbb{E}\mathbbm{1}(o\in \mathfrak{S})|\mathcal{C}_o|^{-1} \sum_{v\in \mathcal{C}_o} \nu_{G[\mathfrak{S}],v}(\{\theta\})=\mathbb{E}\mathbbm{1}(o\in \mathfrak{S})|\mathcal{C}_o|^{-1}.\]

Thus, from the Mass-Transport Principle, we have
\[\mathbb{E} \nu_{G[\mathfrak{S}],o}(\{\theta\})=\mathbb{E}\mathbbm{1}(o\in \mathfrak{S})|\mathcal{C}_o|^{-1}.\] 

The last statement of the theorem about the case $\theta=0$ is the same as Theorem~\ref{thm:GE_0} part \eqref{GallaiEdmondsd}, which we will prove below. 
\end{proof}

The next lemma will be used to overcome the difficulty that the Boltzmann random matching at temperature zero might no be unique.

\begin{lemma}\label{welld2}
Let $e=(u,v)$ be an edge such that $u$ is  not 0-neutral. Then  \break$\lim_{t\to 0}s_{G,\{u,v\}}(it)$ exists and finite.

Consequently, if $\mathcal{M}$ is any Boltzmann random matching at temperature zero, then we have
\[\mathbb{P}(e\in\mathcal{M})=-\lim_{t\to 0}s_{G,\{u,v\}}(it),\] 
that is, the probability of $e\in\mathcal{M}$ does not depend on the choice of $\mathcal{M}$.
\end{lemma}
\begin{proof}
First assume that $u$ is positive. Then the limits $\lim_{t\to 0}\frac{s_{G,u}(it)}{it}$ and \break  $\lim_{t\to 0}it\cdot s_{G-u,v}(it)$ exist and they are finite. Thus, \[\lim_{t\to 0}s_{G,\{u,v\}}(it)=\lim_{t\to 0}\frac{s_{G,u}(it)}{it}\cdot \left(it\cdot s_{G-u,v}(it)\right)\]
exists and finite.

\begin{samepage}Now assume that $u$ is essential. If $v$ is positive in $G-u$, then $\lim_{t\to 0}\frac{s_{G-u,v}(it)}{it}$ exists and finite. So
\[\lim_{t\to 0}s_{G,\{u,v\}}(it)=\lim_{t\to 0}it\cdot s_{G,u}(it) \cdot \frac{s_{G-u,v}(it)}{it}\]
exists and finite.
\end{samepage}

Thus, it is enough to prove that if $u$ is essential, then $v$ must be positive in $G-u$. For the sake of contradiction, assume that $v$ is not positive $G-u$. Then  $\lim_{t\to 0}\frac{s_{G-u,v}(it)}{it}=-\infty$, so $\lim_{t\to 0}s_{G,\{u,v\}}(it)=-\infty$. This is a contradiction, since $-s_{G,\{u,v\}}(it)=\mathbb{P}(e\in \mathcal{M}_G^t)\in [0,1]$.

The second statement of the lemma follows from the first one, Lemma~\ref{postemp}  and the definition of a Boltzmann random matching  at temperature zero.
\end{proof}

In the next lemma, we prove Theorem~\ref{thm:GE_0} part \eqref{GallaiEdmondsd}.
\begin{lemma}
Let $(G,o)$ be a unimodular random rooted graph.  Then
\[\mathbb{E}\nu_{G,o}(\{0\})=\mathbb{E}\mathbbm{1}(o\in D) |\mathcal{C}_o|^{-1}-\mathbb{P}(o\in A).\]

Here, for  $o\in D$, $\mathcal{C}_o$ is the connected component of $o$ in the graph $G[D]$.
\end{lemma}
\begin{proof}
Let $\mathcal{M}$ be a Boltzmann random matching at temperature zero on $G$. Let us define
\begin{multline*}f(G,x,y)\\=\mathbbm{1}(x\in A\text{ and } y\in D)\mathbb{P}(x\text{ is matched by $\mathcal{M}$ with a vertex in }\mathcal{C}_y)|\mathcal{C}_y|^{-1}.
\end{multline*}

Although we might not have a canonical choice for the Boltzmann random matching $\mathcal{M}$, the function $f(G,x,y)$ is still well defined. Indeed, if $x\in A$, then $x$ is positive, thus Lemma \ref{welld2} can be used to see that the probability above does not depend on the choice of $\mathcal{M}$.

If $o\in A$, then $o$ is matched with a vertex in $D$ with probability $1$, as it follows from Theorem~\ref{thm:GE_0} part \eqref{GallaiEdmondsc}. Thus,
\[\mathbb{E}\sum_{v\in V(G)}f(G,o,v)=\mathbb{P}(o\in A).\]
Again, by Theorem~\ref{thm:GE_0} part \eqref{GallaiEdmondsc}, we know that for $o\in D$, either 
\begin{enumerate}[(i)]
    \item $\mathcal{M}$ leaves exactly one vertex of $\mathcal{C}_o$ uncovered and no edge of $\mathcal{M}$ connects a vertex of $\mathcal{C}_o$ with a vertex of $A$; or
    \item $\mathcal{M}$ covers $\mathcal{C}_o$ and there is exactly one edge in $\mathcal{M}$ that connects a vertex of $\mathcal{C}_o$ with a vertex of $A$.
\end{enumerate}
So
\begin{align*}\mathbb{E}\sum_{v\in V(G)}f(G,v,o)&=\mathbb{E}\mathbbm{1}(o\in D) \mathbb{P}(\mathcal{M}\text{ covers }\mathcal{C}_o) |\mathcal{C}_o|^{-1}\\&=
\mathbb{E}\mathbbm{1}(o\in D)|\mathcal{C}_o|^{-1}\left(1-\sum_{w\in \mathcal{C}_o}\mathbb{P}(w\text{ is not covered by }\mathcal{M})\right).
\end{align*}
Again, note that the probability of the event that $w\text{ is not covered by }\mathcal{M}$ does not depend on the choice of $\mathcal{M}$ by Lemma \ref{welld}.

Then from the Mass-Transport Principle, we have
\begin{equation}\label{eqmtp1}
\mathbb{P}(o\in A)=\mathbb{E}\mathbbm{1}(o\in D)|\mathcal{C}_o|^{-1}\left(1-\sum_{w\in \mathcal{C}_o}\mathbb{P}(w\text{ is not covered by }\mathcal{M})\right).
\end{equation}
Now, let
\[g(G,x,y)=\mathbbm{1}(x\in D\text{ and }y\in \mathcal{C}_x) |\mathcal{C}_x|^{-1} \mathbb{P}(x\text{ is not covered by }\mathcal{M}).\]

Again, this is well defined by Lemma \ref{welld}.

Then
\begin{align*}
\mathbb{E}\sum_{v\in V(G)} g(G,o,v)&=\mathbb{E} \mathbb{P}(o\in D \text{ and } o\text{ is not covered by }\mathcal{M})\\&=\mathbb{E} \mathbb{P}( o\text{ is not covered by }\mathcal{M}),
\end{align*}
where we used the fact that if $o\not\in D$, then  $o$ is covered by $\mathcal{M}$ with probability~$1$. Moreover, we have
\[\mathbb{E}\sum_{v\in V(G)} g(G,o,v)=\mathbb{E}\mathbbm{1}(o\in D)|\mathcal{C}_o|^{-1}\sum_{w\in \mathcal{C}_o} \mathbb{P}(w\text{ is not covered by }\mathcal{M}).\]

Thus, from the Mass-Transport Principle, we have
\begin{multline*}\mathbb{E}\mathbbm{1}(o\in D)|\mathcal{C}_o|^{-1}\sum_{w\in \mathcal{C}_o} \mathbb{P}(w\text{ is not covered by }\mathcal{M})=\mathbb{E} \mathbb{P}( o\text{ is not covered by }\mathcal{M}).
\end{multline*}
Inserting this into Equation \eqref{eqmtp1}, we obtain that 
\begin{align*}
\mathbb{P}(o\in A)&=\mathbb{E}\mathbbm{1}(o\in D)|\mathcal{C}_o|^{-1}\left(1-\sum_{w\in \mathcal{C}_o}\mathbb{P}(w\text{ is not covered by }\mathcal{M})\right)\\&=\mathbb{E}\mathbbm{1}(o\in D)|\mathcal{C}_o|^{-1}-\mathbb{E}\mathbb{P}( o\text{ is not covered by }\mathcal{M})\\&=\mathbb{E}\mathbbm{1}(o\in D)|\mathcal{C}_o|^{-1}-\mathbb{E}\nu_{G,o}(\{0\}),  
\end{align*}
and this is exactly what we needed to prove.
\end{proof}


 Salez~\cite{salez}  defined   the \emph{tree-complexity} $\tau(\theta)$ of a  totally real algebraic integer $\theta\in \mathbb{A}$, as the size of the smallest tree, such that $\theta$ is a root of its characteristic polynomial. Note that for every totally real algebraic integer~$\theta$, we have $\tau(\theta)<\infty$, that is, there is a finite tree  such that $\theta$ is a root of its characteristic polynomial~\cite{salez2015every}.   Similarly, we define the \emph{matching-complexity} $\tau_m(\theta)$ of a $\theta\in \mathbb{A}$ as the size of the smallest graph, such that $\theta$ is a root of its matching polynomial. Since for any tree the characteristic polynomial and the matching polynomial coincide, we have $\tau_m(\theta)\le\tau(\theta)$.

The isoperimetric constant $i(G)$ of a graph $G$ is defined as 
\[i(G)=\inf \left(\frac{|\partial S|}{|S|} \quad \Big| \quad \emptyset\neq S\subseteq  V(G),\quad |S|<\infty\right).\]

Thereom \ref{thm25} is an easy consequence of the following theorem.
\begin{theorem}\label{thm3}
Let $(G,o)$ be a unimodular random rooted graph  with maximum degree at most $D$. Assume that there is an $h>0$ such that $i(G)\ge h$ with probability~$1$. Then $\mathbb{E}\nu_{G,o}$ has only finitely many atoms. 
\end{theorem}

\begin{proof}
We follow the approach of Salez \cite{salez}. 

It is enough to show that the matching-complexity of the atoms are bounded. Let $\theta\in\mathbb{A}$ be an atom of the expected matching measure. Using the notations of Theorem~\ref{thm2}, we have that
\begin{align}\label{ineq2}
    0<\mathbb{E}\nu_{G,o}(\{\theta\})&\le \mathbb{E}\mathbbm{1}(o\in\mathfrak{S})|\mathcal{C}_o|^{-1}-\mathbb{P}(o\in \partial \mathfrak{S})\nonumber \\
    &\le \mathbb{P}(o\in \mathfrak{S})\frac{1}{\tau_m(\theta)}-\mathbb{P}(o\in\partial\mathfrak{S}).
\end{align}
Let us define 
\[f(G,x,y)=\mathbbm{1}(x\in \mathfrak{S}\text{ and }y\in \partial \mathcal{C}_x)|\mathcal{C}_x|^{-1}.\]
Then 
\[\mathbb{E}\sum_{v\in V(G)}f(G,o,v)=\mathbb{E}\mathbbm{1}(o\in \mathfrak{S})\frac{|\partial \mathcal{C}_o|}{|\mathcal{C}_o|}\ge h \mathbb{P}(o\in \mathfrak{S}).\]
Moreover,
\begin{align*}
\mathbb{E}\sum_{v\in V(G)}f(G,v,o)&=\mathbb{E}\mathbbm{1}(o\in \partial \mathfrak{S}) \left|\{\mathcal{C}_w|w\in \partial \{o\}\cap \mathfrak{S}\}\right|\\&\le \mathbb{E}\mathbbm{1}(o\in \partial \mathfrak{S}) \deg(o)\\&\le D\mathbb{P}(o\in \partial\mathfrak{S}).
\end{align*}
Thus, using Mass-Transport Principle, we have 
\[
  D\cdot \mathbb{P}(o\in \partial\mathfrak{S})\ge h\cdot\mathbb{P}(o\in \mathfrak{S}).
\]
Combining this with Inequality \eqref{ineq2}, we obtain
\[
    \tau_m(\theta)\le \frac{\mathbb{P}(o \in\mathfrak{S})}{\mathbb{P}(o \in\partial \mathfrak{S})}\le \frac{D}{h}.
\]
\end{proof}
\begin{reptheorem}{thm25}
Let $(G,o)$ be an ergodic non-amenable random rooted graph  with maximum degree at most $D$. Then $\mathbb{E}\nu_{G,o}$ has only finitely many atoms. \end{reptheorem}
\begin{proof}
Ergodicity gives us that $i(G)$ is constant almost surely. As $(G,o)$ is non-amenable, this constant must be positive. Thus, Theorem \ref{thm3} can be applied.  
\end{proof}

\section{Further remarks and open questions}\label{sec:open}

\begin{question}
Let $A$ be a subset of the vertices of a graph $G$, such that all the vertices in $A$ are special. Let $u$ be a neutral vertex. Is it true that $u$ is neutral in $G-A$?
\end{question}
Note that, if $U$ is finite, then we have an affirmative answer for the question above, because we can apply Lemma~\ref{neustab} and Lemma~\ref{specstab} iteratively. 

\begin{question}
Can we replace the inequality in Theorem~\ref{thm2}  with equality?
\end{question}

The \emph{anchored isoperimetric constant} $i^\star(G,o)$ of a rooted graph $(G,o)$ is defined as

 \begin{multline*}
 i^\star(G,o) =  \lim_{n\to\infty} \inf\left\{\frac{|\partial S|}{|S|}\colon o\in S\subseteq V(G),\, G[S] \textrm{ is connected},\, n\le |S|< \infty\right\}.
 \end{multline*}
Observe that $i^\star(G,o)\le i(G)$. Salez~\cite{salez} proved the following theorem.
\begin{theorem}[Salez~\cite{salez}] Let $h>0$. Assume that $(G,o)$ is a unimodular random rooted graph, such that with probability $1$, the graph $G$ is a tree, the minimum degree of $G$ is at least $2$, the maximum degree of $G$ is at most $D$ and $i^\star(G,o)\ge h$. Then $\mathbb{E}\nu_{G,o}$ has only finitely many atoms.
\end{theorem}

\begin{question}
Is there some version of Theorem~\ref{thm3} with anchored expansion (for graphs which are not necessarily trees)?
\end{question}

The next question was already mentioned in Subsection \ref{sec:monomerdimer}.
\begin{question}\label{qmonomerdimer}
Is it true that the random matchings $\mathcal{M}_G^t$ converge in law as $t\to 0$?
\end{question}

We believe that the answer to the question above should be negative as we explain now. Consider two disjoint copies $(G,o)$ and $(G',o')$ of the same rooted tree. Connect them with the edge $oo'$ to obtain the graph $H$. Then one can verify that \[\mathbb{P}(oo'\in \mathcal{M}_H^t)=\frac{1}{1-\left(s_{G,o}(it)\right)^{-2}}.\]
Thus, to prove that the random matchings $\mathcal{M}_H^t$ do not converge in law for an appropriately  chosen rooted tree $(G,o)$, it is enough to show that $s_{G,o}(it)$ has no limit (neither finite, nor infinite) as $t\to 0$. One can find measures such that their Stieltjes-transform satisfies this property\footnote{One needs to search among measures such that their CDF is not differentiable at $0$, as \cite[Theorem 2.1]{silvs} suggests. However, we need to be careful, because having a radial limit along the line $\{it\}$ is not exactly the same as the condition of \cite[Theorem 2.1]{silvs}.}, and we are unaware of any results saying that these measures can not be obtained as a spectral measure of a rooted tree. However, we do not know how to construct such a tree.   

Let us recall the following definition from the paper of Coste and Salez~\cite{coste}. We say a measure $\nu$ has no extended states at a location $E$, if
\[\lim_{\varepsilon\to 0+}\frac{\nu([E-\varepsilon,E+\varepsilon])-\nu(\{E\})}{\varepsilon}=0.\]

\begin{proposition}
Let $u$ be a $\theta$-positive vertex of $G$. Then $\nu_{G,u}$ has no extended states at $\theta$.
\end{proposition}
\begin{proof}
Since $u$ is $\theta$-positive, we have that  $\lim_{t\to 0}\frac{s_{G,u}(\theta+it)}{it}$ is finite. In particular,  $\lim_{t\to 0} s_{G,u}(\theta+it)=0$. Observe, that
\begin{align*}
\nu_{G,u}([\theta-t,\theta+t])&=\int_{\theta-t}^{\theta+t} 1d\nu_{G,u}(x)\\&\le \int_{\theta-t}^{\theta+t} \frac{2t^2}{(\theta-x)^2+t^2}d\nu_{G,u}(x)\\&\le
\int_{-D}^D \frac{2t^2}{(\theta-x)^2+t^2}d\nu_{G,u}(x)\\&=
-2t\cdot \text{Im} \int_{-D}^D \frac{1}{\theta+it-x}d\nu_{G,u}(x)\\&=-2t\cdot  \text{Im} s_{G,u}(\theta+it),
\end{align*}
which shows that $\nu_{G,u}$ has no extended states at $\theta$. 
\end{proof}

\begin{question}
Let us consider the $d$-dimensional grid $\mathbb{Z}^d$, let $o$ be any vertex of it. Is $o$ $0$-neutral or $0$-positive? Does $\nu_{\mathbb{Z}^d,o}$ has extended states at $0$ or not? 
\end{question}

Note that for $d=1$, we have that $o$ is $0$-neutral, and we have extended states at~$0$.

Note that for $d$-regular trees, all the vertices are $0$-neutral. It seems difficult to understand the matching measure of graphs, which are not tree. In particular, the following question is still open.

\begin{question}
Is there an infinite vertex-transitive graph such that all the vertices are $0$-positive?
\end{question}

We proved in Theorem~\ref{lemmafinite} that all unimodular critical graphs are finite. Now we give an example of an infinite connected $\theta$-critical graph $G$ for $\theta=\frac{5}{2}$. The construction is the following: Take $5$ half-infinite paths, which start from the same vertex $o$, but they are disjoint otherwise. Let $f\in \ell^2(V(G))$ be defined as follows. For any vertex $v$ of $G$, let $f(v)=2^{-\ell(v)}$, where $\ell(v)$ is the distance of $v$ from $o$. It is straightforward to check that $f$ is a  nowhere vanishing vector in the $\theta$-eigenspace of the adjacency operator of $G$. Since $G$ is a tree, this easily implies that $G$ is $\theta$-critical. However, we still do not know the answer for the following question.

\begin{question}
Is there an infinite connected $0$-critical graph?
\end{question}

Note that a $0$-critical graph can not be a tree. In fact, it can not be bipartite as we show next. 

\begin{lemma}
Let $P$ be a path with odd number of edges in the graph $G$. Then $P$ is not $0$-essential.
\end{lemma}
\begin{proof}
Let $u$ be the start vertex of $P$. Let $\mathcal{P}^{\text{odd}}\subseteq  \mathcal{P}(u)$ be the set of paths starting from $u$ with odd number of edges. Let $\mathcal{P}^{\text{even}}=\mathcal{P}(u)\backslash \mathcal{P}^{\text{odd}}$. Let $H^{\text{odd}}$ be the closed subsapce of $\ell^2(\mathcal{P}(u))$ consisting of vectors  such that their support is contained in $\mathcal{P}^{\text{odd}}$. We define $H^{\text{even}}$ in an analogous way.  Let $A$ be the adjacency operator of $T(G,u)$. Then the $0$-eigenspace of $A$ is just $\ker A$. It is easy to see that $\ker A$ is the orthogonal direct sum of $\ker A\cap H^{\text{odd}}$ and $\ker A\cap H^{\text{even}}$.
As before, let $\Pi_{G,u}$ be the orthogonal projection to $\ker A$. Then it is clear from what is written above that $\Pi_{G,u}\chi_u$ is supported on $\mathcal{P}^{\text{even}}$ as $\chi_u\in H^{\text{even}}$. Therefore, $\langle \Pi_{G,u} \chi_u,\chi_P\rangle=0$, that is, $P$ is not $0$-essential. 
\end{proof}

\begin{lemma}
Let $G$ be a bipartite graph. Let $u$ be an essential vertex of it. Then all the neighbours of $u$ are special. 
\end{lemma}
\begin{proof}
Let $v$ be a neighbor of $u$. It is enough to show that $u$ is essential in $G-v$. Since $G$ is bipartite all the paths in $\mathcal{P}(u,v)$ have an odd number of edges, thus, from the previous lemma, they are all not $0$-essential. Therefore, Lemma~\ref{pathlemma} can be applied to give us that $u$ is essential in $G-v$. 
\end{proof}
As an easy corollary, we get the following.
\begin{lemma}
Let $G$ be a connected bipartite graph with at least two vertices, then $G$ is not $0$-critical.
\end{lemma}

\bibliography{references}
\bibliographystyle{plain}

\end{document}